\documentclass{amsart}

\usepackage{pstricks,pst-node,graphicx}
\usepackage{amsmath,amsfonts,amssymb}
\usepackage[backref=page]{hyperref}
\usepackage{url, hypcap}
\hypersetup{colorlinks=true, citecolor=darkblue, linkcolor=darkblue}
\definecolor{darkblue}{rgb}{0.0,0,0.7} % darkblue color
\definecolor{darkred}{rgb}{0.7,0,0} % darkred color
\newcommand{\defn}[1]{{\color{darkred}\emph{#1}}} % emphasis of a definition

\usepackage[colorinlistoftodos]{todonotes}

\newtheorem{theorem}{Theorem}[section]
\newtheorem{lemma}[theorem]{Lemma}
\newtheorem{proposition}[theorem]{Proposition}
\newtheorem{corollary}[theorem]{Corollary}

\theoremstyle{definition}
\newtheorem{definition}[theorem]{Definition}
\newtheorem{example}[theorem]{Example}

\theoremstyle{remark}
\newtheorem{remark}[theorem]{Remark}

\numberwithin{equation}{section}

\newcommand{\SSKT}{\ensuremath\mathrm{SSKT}}
\newcommand{\SSYT}{\ensuremath\mathrm{SSYT}}
\newcommand{\RF}{\ensuremath\mathrm{RF}}
\newcommand{\RFC}{\ensuremath\mathrm{RFC}}
\newcommand{\wt}{\ensuremath\mathrm{wt}}

\newcommand{\schubert}{\ensuremath\mathfrak{S}}
\newcommand{\key}{\ensuremath\kappa}

\newlength\cellsize 
\savebox2{%
\begin{picture}(12,12)
\put(0,0){\line(1,0){12}}
\put(0,0){\line(0,1){12}}
\put(12,0){\line(0,1){12}}
\put(0,12){\line(1,0){12}}
\end{picture}}

\newcommand\cellify[1]{\def\thearg{#1}\def\nothing{}%
\ifx\thearg\nothing
\vrule width0pt height\cellsize depth0pt\else
\hbox to 0pt{\usebox2\hss}\fi%
\vbox to 12\unitlength{
\vss
\hbox to 12\unitlength{\hss$#1$\hss}
\vss}}

\newcommand\tableau[1]{\vtop{\let\\=\cr
\setlength\baselineskip{-12000pt}
\setlength\lineskiplimit{12000pt}
\setlength\lineskip{0pt}
\halign{&\cellify{##}\cr#1\crcr}}}

\savebox4{% CIRCLE
\begin{picture}(12,12)
\put(6,6){\circle{12}}
\end{picture}}

\newcommand{\cir}[1]{\def\thearg{#1}\def\nothing{}%
\ifx\thearg\nothing\vrule width0pt height\cellsize depth0pt%
  \else\hbox to 0pt{\usebox4\hss}\fi%
  \vbox to 12\unitlength{\vss\hbox to 12\unitlength{\hss$#1$\hss}\vss}}

\begin{document}

%%%%%%%%%%%%%%%%%%%%%%%%%%%%%%%%%%%%%%%%%%%%%%%%%%%%%%%%%%%%
%  TITLE PAGE information
%%%%%%%%%%%%%%%%%%%%%%%%%%%%%%%%%%%%%%%%%%%%%%%%%%%%%%%%%%%%

%     [Short Title]{Full Title}
\title[Demazure crystals for Schubert polynomials]{A Demazure crystal construction for Schubert polynomials}  

\author[S. Assaf]{Sami Assaf}
\address{Department of Mathematics, University of Southern California, 3620 S. Vermont Ave., Los Angeles, CA 90089-2532, U.S.A.}
\email{shassaf@usc.edu}
%\thanks{}

\author[A. Schilling]{Anne Schilling}
\address{Department of Mathematics, UC Davis, One Shields Ave., Davis, CA 95616-8633, U.S.A.}
\email{anne@math.ucdavis.edu}
%\thanks{}

%    General info

%\date{\today}

%\dedicatory{}


\begin{abstract}
Stanley symmetric functions are the stable limits of Schubert polynomials. In this paper, we show that, conversely, Schubert polynomials are Demazure truncations of Stanley symmetric functions. This parallels the relationship between Schur functions and Demazure characters for the general linear group. We establish this connection by imposing a Demazure crystal structure on key tableaux, recently introduced by the first author in connection with Demazure characters and Schubert polynomials, and linking this to the type A crystal structure on reduced word factorizations, recently introduced by Morse and the second author in connection with Stanley symmetric functions. 
\end{abstract}

\maketitle
%\tableofcontents
%\listoffigures

%%%%%%%%%%%%%%%%%%%%%%%%%%%%%%%%%%%%%%%%%%%%%%%%%%%%%%%%%%%%%%%%
%
\section{Introduction}
%
%%%%%%%%%%%%%%%%%%%%%%%%%%%%%%%%%%%%%%%%%%%%%%%%%%%%%%%%%%%%%%%%
\label{sec:introduction}

Schubert polynomials $\schubert_w$ were first introduced by Bernstein et al.~\cite{BGG.1973} as certain polynomial 
representatives of cohomology classes of Schubert cycles $X_w$ in flag varieties. They were extensively studied by 
Lascoux and Sch\"utzenberger~\cite{LS.1982} using an explicit definition in terms of difference operators $\partial_w$.
 Subsequently, a combinatorial expression for Schubert polynomials as the generating polynomial for compatible sequences 
 for reduced expressions of a permutation $w$ was discovered by Billey, Jockusch, and Stanley~\cite{BJS93}. In the special 
 case of the Grassmannian subvariety, Schubert polynomials are Schur polynomials, which also arise as the irreducible 
 characters for the general linear group.

The Stanley symmetric functions $F_w$ were introduced by Stanley~\cite{Stanley.1984} in the pursuit of enumerations of 
the reduced expressions of permutations, in particular of the long permutation $w_0$. They are defined combinatorially as 
the generating functions of reduced factorizations of permutations. Stanley symmetric functions are the stable limit of 
Schubert polynomials~\cite{Mac91,Macdonald.1991}, precisely
\begin{equation}
  F_{w}(x_1,x_2,\ldots) = \lim _{m\to \infty } \mathfrak{S}_{1^{m}\times w} (x_1,x_2,\ldots,x_{n+m}).
  \label{e:stable}
\end{equation}

Edelman and Greene~\cite{EG.1987} showed that the coefficients of the Schur expansion of Stanley symmetric functions 
are nonnegative integer coefficients. 

Demazure modules for the general linear group \cite{Demazure.1974} are closely related to Schubert classes for 
the cohomology of the flag manifold. In certain cases these modules are irreducible polynomial representations, and 
so the Demazure characters also contain the Schur polynomials as a special case. Lascoux and 
Sch\"utzenberger~\cite{LS.1985} stated that Schubert polynomials are nonnegative sums of Demazure characters.
This was proven by Reiner and Shimozono~\cite{ReinerShimozono.1995} using the right keys associated to 
Edelman--Greene insertion.
Using a key tableaux interpretation for Demazure characters~\cite{Ass-W}, Assaf~\cite{Ass-R} showed that the 
Edelman and Greene algorithm giving the Schur expansion of a Stanley symmetric function can be modified to 
a weak Edelman--Greene algorithm which gives the Demazure expansion of a Schubert polynomial.

In this paper, we deepen this connection and provide a converse to \eqref{e:stable} by showing that Schubert polynomials 
are Demazure truncations of Stanley symmetric functions. Specifically, we show in Theorem~\ref{theorem.main} that the 
combinatorial objects underlying the Schubert polynomials, namely the compatible sequences, exhibit a Demazure crystal 
truncation of the full Stanley crystal of Morse and Schilling~\cite{MS16}. We prove this using 
Theorem~\ref{theorem.key demazure}, in which we give an explicit Demazure crystal structure on semi-standard key tableaux, 
which coincide with semi-skyline augmented fillings of Mason~\cite{Mason.2009}. This, together with Theorem~\ref{theorem.weak EG}, 
in which we show that the crystal operators on reduced 
factorizations intertwine with (weak) Edelman--Greene insertion, proves our main result.

Lenart~\cite{Lenart.2004} defined crystal operators on RC graphs \cite{BergeronBilley.1993}, which are closely 
related to compatible sequences, though it was not observed there that this structure is a Demazure crystal. 
Earlier, Reiner and Shimozono~\cite{ReinerShimozono.1995a} defined $r$-pairings on factorized row-frank words that 
can now be interpreted as crystal operators, but again, this was not observed nor was it noted that this structure is a 
Demazure crystal structure. One could complete either of these perspectives to prove our main result, though we prefer 
the key tableaux approach given its simplicity, the natural crystal operators on these objects, and the connection with 
Edelman--Greene insertion.

This paper is structured as follows. In Section~\ref{sec:Tab}, we review the crystal structure on semi-standard
Young tableaux and define Demazure crystals. In Section~\ref{sec:Key}, we introduce new crystal operators
on key tableaux and prove that this amounts to a Demazure crystal (Theorem~\ref{theorem.key demazure}).
Section~\ref{sec:stanley} is reserved for the review of Stanley symmetric functions, Edelman--Greene insertion and
the crystal structure on reduced factorization, which underly the Stanley symmetric functions.
Section~\ref{sec:RCF} contains our main result (Theorem~\ref{theorem.main}), namely a Demazure crystal structure 
on reduced factorizations with cutoff, which are equivalent to compatible sequences. This gives a Demazure
crystal structure for Schubert polynomials and shows that Schubert polynomials are a Demazure truncation of
Stanley symmetric functions.

%%%%%%%%%%%%%%%%%%%%%%%%%%%%%%%%%%%%%%%%%%%%%%%%%%%%%%%%%%%%%%
\subsection*{Acknowledgments}
AS was partially supported by NSF grant  DMS--1500050.
The authors are grateful to Per Alexandersson, Sara Billey, Jim Haglund, Cristian Lenart, Sarah Mason, Liz Milicevic, Jennifer Morse, Vic Reiner, Mark Shimozono, and Alex Yong for helpful discussions and comments on this topic.
AS would also like to thank the University of Southern California for their hospitality during her talk in March
2017 and the AWM Research Symposium at UCLA in April 2017, where this work started.

%%%%%%%%%%%%%%%%%%%%%%%%%%%%%%%%%%%%%%%%%%%%%%%%%%%%%%%%%%%%%%%%
%
\section{Crystal structure on tableaux}
%
%%%%%%%%%%%%%%%%%%%%%%%%%%%%%%%%%%%%%%%%%%%%%%%%%%%%%%%%%%%%%%%%
\label{sec:Tab}

We begin in Section~\ref{sec:schur} by reviewing the basics of Schur polynomials via the combinatorics of Young 
tableaux. In Section~\ref{sec:SSYT}, we review the type $A$ crystal structure on semi-standard Young tableaux, and 
conclude in Section~\ref{sec:demazure crystal} with the definition of Demazure crystals.

%%%%%%%%%%%%%%%%%%%%%%%%%%%%%%%%%%%%%%%%%%%%%%%%%%%%%%%%%%%%%%%%
%
\subsection{Combinatorics of Schur polynomials}
\label{sec:schur}

Given a partition $\lambda$, the \defn{Young diagram of shape $\lambda$} is the array of left-justified cells with 
$\lambda_i$ boxes in row $i$. Here we use French notation, where the rows weakly decrease in size from bottom 
to top in the Young diagram. A \defn{Young tableau} is a filling of the cells of a Young diagram from some totally 
ordered alphabet (for example the set of positive integers) such that rows and columns weakly increase. A 
\defn{semi-standard Young tableau} is a Young tableau with distinct column entries. Figure~\ref{fig:SSYT}
provides an example of semi-standard Young tableaux of a fixed shape.

\begin{figure}[ht]
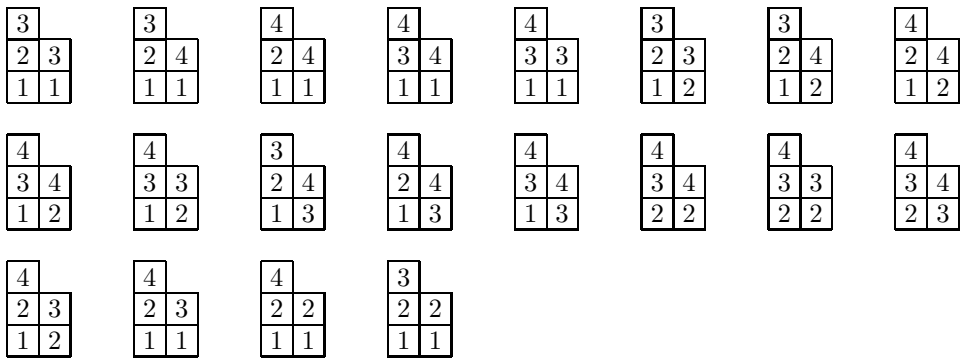

  \begin{displaymath}
    \begin{array}{c@{\hskip 2\cellsize}c@{\hskip 2\cellsize}c@{\hskip 2\cellsize}c@{\hskip 2\cellsize}c@{\hskip 2\cellsize}c@{\hskip 2\cellsize}c@{\hskip 2\cellsize}c}
      \tableau{3 \\ 2 & 3 \\ 1 & 1} &
      \tableau{3 \\ 2 & 4 \\ 1 & 1} &
      \tableau{4 \\ 2 & 4 \\ 1 & 1} &
      \tableau{4 \\ 3 & 4 \\ 1 & 1} &
      \tableau{4 \\ 3 & 3 \\ 1 & 1} &
      \tableau{3 \\ 2 & 3 \\ 1 & 2} &
      \tableau{3 \\ 2 & 4 \\ 1 & 2} &
      \tableau{4 \\ 2 & 4 \\ 1 & 2} \\ \\
      \tableau{4 \\ 3 & 4 \\ 1 & 2} &
      \tableau{4 \\ 3 & 3 \\ 1 & 2} &
      \tableau{3 \\ 2 & 4 \\ 1 & 3} &
      \tableau{4 \\ 2 & 4 \\ 1 & 3} &
      \tableau{4 \\ 3 & 4 \\ 1 & 3} &
      \tableau{4 \\ 3 & 4 \\ 2 & 2} &
      \tableau{4 \\ 3 & 3 \\ 2 & 2} &
      \tableau{4 \\ 3 & 4 \\ 2 & 3} \\ \\
      \tableau{4 \\ 2 & 3 \\ 1 & 2} &
      \tableau{4 \\ 2 & 3 \\ 1 & 1} &
      \tableau{4 \\ 2 & 2 \\ 1 & 1} &
      \tableau{3 \\ 2 & 2 \\ 1 & 1} & & & &
    \end{array}
  \end{displaymath}
  \caption{\label{fig:SSYT}The semi-standard Young tableaux of shape $(2,2,1)$ over the alphabet $\{1,2,3,4\}$.}
\end{figure}

The \defn{weight} of a semi-standard Young tableau $T$, denoted by $\wt(T)$, is the weak composition 
whose $i$th part is the number of occurrences of $i$ in $T$. The shape $\lambda$ of $T$ is also denoted by
$\operatorname{sh}(T)$.

\begin{definition}
  The \defn{Schur polynomial} in $n$ variables indexed by the partition $\lambda$ is
  \begin{equation}
    s_\lambda(x) = s_{\lambda}(x_1,\ldots,x_n) = \sum_{T \in \SSYT_n(\lambda)} x_1^{\wt(T)_1} \cdots x_n^{\wt(T)_n},
  \end{equation}
  where $\SSYT_n(\lambda)$ is the set of semi-standard Young tableaux of shape $\lambda$ over
  the alphabet $\{1,2,\ldots,n\}$.
  \label{def:schur}
\end{definition}

Schur polynomials arise as characters for irreducible highest weight modules for the general linear group with 
semi-standard Young tableaux giving a natural indexing set for the basis of the module.

%%%%%%%%%%%%%%%%%%%%%%%%%%%%%%%%%%%%%%%%%%%%%%%%%%%%%%%%%%%%%%%%
%
\subsection{Crystal operators on semi-standard Young tableaux}
\label{sec:SSYT}

A crystal graph is a directed, colored graph with vertex set given by the crystal basis and directed edges given 
by deformations of the Chevalley generators. For the quantum group $U_{q}(\mathfrak{sl}_n)$, the crystal basis 
can be indexed by semi-standard Young tableaux over the alphabet $A=\{1,2,\ldots,n\}$ and there is an explicit 
combinatorial construction of the crystal graph on tableaux~\cite{KN94,Lit95}. For an introduction to crystals
from the quantum group perspective, see~\cite{HongKang.2002}. For a purely combinatorial introduction to 
crystals, see \cite{Bump.Schilling.2017}.

For a word $w$ of length $k$ with letters from the alphabet $A=\{1,2,\ldots,n\}$, an integer $0 \leqslant r \leqslant k$, 
and an integer $1\leqslant i<n$, define
\begin{equation}
  M_i(w,r) = \wt(w_{1} w_{2} \cdots w_{r})_{i} - \wt(w_{1} w_{2} \cdots w_{r})_{i+1},
\end{equation}
where $\wt(w)$ is the weak composition whose $j$th part is the number of $j$'s in $w$. Set 
$M_i(w) = \max_{r\geqslant 0}\{M_i(w,r)\}$. Observe that if $M_i(w) > 0$ and $p$ is the leftmost occurrence of this 
maximum, then $w_p = i$, and if $q$ is the rightmost occurrence of this maximum, then either $q=k$ or $w_{q+1} = i+1$.

For a Young tableau $T$, the \defn{column reading word of $T$}, denoted by $w(T)$, is the word obtained by 
reading the entries of $T$ down columns from left to right. For example, the column reading word of the leftmost 
Young tableau in the top row of Figure~\ref{fig:SSYT} is $32131$.

\begin{definition}
  Given an integer $1\leqslant i<n$, define the \defn{lowering operator} $f_i$ on semi-standard Young tableaux over
  the alphabet $A$ as follows: 
  if $M_i(w(T)) \leqslant 0$, then $f_i(T)=0$; otherwise, let $p$ be the smallest index such that $M_i(w(T),p) = M_i(w(T))$, 
  and $f_i(T)$ changes the entry in $T$ corresponding to $w(T)_p$ to $i+1$.
  \label{def:young-lower}
\end{definition}

An example of the lowering operator $f_2$ is given in Figure~\ref{fig:young-lower}. For this example, the column reading 
word is given below each semi-standard Young tableau with the largest index that attains $M_2(w(T))>0$ underlined 
and the corresponding entry in the tableau circled.

\begin{figure}[ht]
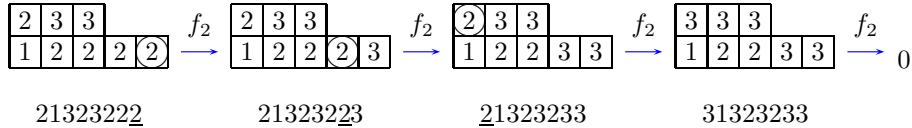

  \begin{displaymath}
    \begin{array}{c@{\hskip 2\cellsize}c@{\hskip 2\cellsize}c@{\hskip 2\cellsize}c@{\hskip 2\cellsize}c}
      \rnode{A}{\tableau{2 & 3 & 3 \\ 1 & 2 & 2 & 2 & \cir{2} \\ & }} &
      \rnode{B}{\tableau{2 & 3 & 3 \\ 1 & 2 & 2 & \cir{2} & 3 \\ & }} &
      \rnode{C}{\tableau{\cir{2} & 3 & 3 \\ 1 & 2 & 2 & 3 & 3 \\ & }} &
      \rnode{D}{\tableau{3 & 3 & 3 \\ 1 & 2 & 2 & 3 & 3 \\ & }} &
      \rnode{E}{\raisebox{-\cellsize}{$0$}} \\
      2132322\underline{2} &
      213232\underline{2}3 &
      \underline{2}1323233 &       
      31323233 &       
    \end{array}
    \psset{nodesep=5pt,linewidth=.1ex}
    \ncline[linecolor=blue]{->} {A}{B} \naput{f_2}
    \ncline[linecolor=blue]{->} {B}{C} \naput{f_2}
    \ncline[linecolor=blue]{->} {C}{D} \naput{f_2}
    \ncline[linecolor=blue]{->} {D}{E} \naput{f_2}
  \end{displaymath}
  \caption{\label{fig:young-lower}An example of the lowering operator $f_2$ on semi-standard Young tableaux.}
\end{figure}

\begin{definition}
  Given an integer $1\leqslant i<n$, define the \defn{raising operator} $e_i$ on semi-standard Young tableaux over
  the alphabet $A$ as follows: 
  let $q$ be the largest index such that $M_i(w(T),q) = M_i(w(T))$. If $q$ is the length of $w(T)$, then $e_i(T)=0$; 
  otherwise, $e_i(T)$ changes the entry in $T$ corresponding to $w(T)_{q+1}$ to $i$.
  \label{def:young-raise}
\end{definition}

For further examples of raising and lowering operators on semi-standard Young tableaux, see Figure~\ref{fig:Young-221}. 
Note that we have drawn the crystal in Figure~\ref{fig:Young-221} with lowering operators pointing upward to facilitate 
the bijection with semi-standard key tableaux as explained in Section~\ref{sec:Key}.

\begin{figure}[ht]
  \begin{displaymath}
    \begin{array}{cccccc}
      & & \rnode{b1}{\tableau{4 \\ 3 & 4 \\ 2 & 3}} & & & \\[3\cellsize]
      & & \rnode{b2}{\tableau{4 \\ 3 & 4 \\ 2 & 2}} & & \rnode{c2}{\tableau{4 \\ 3 & 4 \\ 1 & 3}} & \\[3\cellsize]
      \rnode{a3}{\tableau{4 \\ 3 & 3 \\ 2 & 2}} & & & \rnode{c3}{\tableau{4 \\ 3 & 4 \\ 1 & 2}} & \rnode{C3}{\tableau{4 \\ 2 & 4 \\ 1 & 3}} & \\[3\cellsize]
      & \rnode{b4}{\tableau{4 \\ 3 & 3 \\ 1 & 2}} & \rnode{B4}{\tableau{3 \\ 2 & 4 \\ 1 & 3}} & & \rnode{c4}{\tableau{4 \\ 2 & 4 \\ 1 & 2}} & \rnode{d4}{\tableau{4 \\ 3 & 4 \\ 1 & 1}} \\[3\cellsize]
      & \rnode{b5}{\tableau{4 \\ 2 & 3 \\ 1 & 2}} & \rnode{B5}{\tableau{3 \\ 2 & 4 \\ 1 & 2}} & \rnode{c5}{\tableau{4 \\ 3 & 3 \\ 1 & 1}} & & \rnode{d5}{\tableau{4 \\ 2 & 4 \\ 1 & 1}} \\[3\cellsize]
      \rnode{a6}{\tableau{3 \\ 2 & 3 \\ 1 & 2}} & & & \rnode{c6}{\tableau{4 \\ 2 & 3 \\ 1 & 1}} & \rnode{C6}{\tableau{3 \\ 2 & 4 \\ 1 & 1}} & \\[3\cellsize]
      & & \rnode{B7}{\tableau{3 \\ 2 & 3 \\ 1 & 1 }} & \rnode{c7}{\tableau{4 \\ 2 & 2 \\ 1 & 1 }} & & \\[3\cellsize]
      & & \rnode{B8}{\tableau{3 \\ 2 & 2 \\ 1 & 1}} & & & 
    \end{array}
    \psset{nodesep=2pt,linewidth=.1ex}
    \ncline[linecolor=red]{<-}  {b1}{b2} %\ncput*{2}
    \ncline[linecolor=green]{<-}{b1}{c2} %\ncput*{3}
    \ncline[linecolor=blue]{<-} {b2}{a3} %\ncput*{1}
    \ncline[linecolor=red]{<-}  {c2}{C3} %\ncput*{2}
    \ncline[linecolor=green]{<-}{b2}{c3} %\ncput*{3}
    \ncline[linecolor=blue]{<-} {c3}{b4} %\ncput*{1}
    \ncline[linecolor=blue]{<-} {C3}{B4} %\ncput*{1}
    \ncline[linecolor=red]{<-}  {C3}{c4} %\ncput*{2}
    \ncline[linecolor=green]{<-}{a3}{b4} %\ncput*{3}
    \ncline[linecolor=green]{<-}{c3}{d4} %\ncput*{3}
    \ncline[linecolor=blue]{<-} {c4}{B5} %\ncput*{1}
    \ncline[linecolor=blue]{<-} {d4}{c5} %\ncput*{1}
    \ncline[linecolor=red]{<-}  {b4}{b5} %\ncput*{2}
    \ncline[linecolor=red]{<-}  {B4}{B5} %\ncput*{2}
    \ncline[linecolor=red]{<-}  {d4}{d5} %\ncput*{2}
    \ncline[linecolor=green]{<-}{b4}{c5} %\ncput*{3}
    \ncline[linecolor=green]{<-}{c4}{d5} %\ncput*{3}
    \ncline[linecolor=green]{<-}{b5}{c6} %\ncput*{3}
    \ncline[linecolor=blue]{<-} {B5}{a6} %\ncput*{1}
    \ncline[linecolor=red]{<-}  {c5}{c6} %\ncput*{2}
    \ncline[linecolor=blue]{<-} {d5}{C6} %\ncput*{1}
    \ncline[linecolor=green]{<-}{B5}{C6} %\ncput*{3}
    \ncline[linecolor=red]{<-}  {c6}{c7} %\ncput*{2}
    \ncline[linecolor=blue]{<-} {C6}{B7} %\ncput*{1}
    \ncline[linecolor=green]{<-}{a6}{B7} %\ncput*{3}
    \ncline[linecolor=red]{<-}  {B7}{B8} %\ncput*{2}
    \ncline[linecolor=blue]{<-} {c7}{B8} %\ncput*{1}
  \end{displaymath}
  \caption{\label{fig:Young-221}The crystal $B(2,2,1)$, with edges $f_1 \color{green}\nwarrow$, 
  $f_2 \color{red}\uparrow$, $f_3 \color{blue}\nearrow$.}
\end{figure}

For a partition $\lambda$, we may define the highest weight crystal (of type $A_n$) of highest weight $\lambda$,
denoted $B(\lambda)$, as the set $\SSYT_n(\lambda)$ together with the operators $f_i,e_i$ for $1\leqslant i<n$
and the weight function $\wt$. The character of a crystal is defined as
\[
	\operatorname{ch} B(\lambda) = \sum_{b\in B(\lambda)} x_1^{\wt(b)_1} \cdots x_n^{\wt(b)_n},
\]
which in this case is precisely the Schur polynomial $s_\lambda(x_1,\ldots,x_n)$.

%%%%%%%%%%%%%%%%%%%%%%%%%%%%%%%%%%%%%%%%%%%%%%%%%%%%%%%%%%%%%%%%
%
\subsection{Demazure crystals}
\label{sec:demazure crystal}

Demazure characters first arose in connection with Schubert classes for the cohomology of the flag manifold 
in~\cite{Demazure.1974}.

The divided difference operators $\partial_i$ for $1\leqslant i <n$ act on polynomials by
\[
	\partial_i f(x_1,\ldots,x_n) = \frac{f(x_1,\ldots,x_i,x_{i+1}, \ldots,x_n) - f(x_1,\ldots,x_{i+1},x_i, \ldots,x_n)}
	{x_i-x_{i+1}}.
\]
For $w\in S_n$, we may define $\partial_w = \partial_{i_1} \partial_{i_2} \cdots \partial_{i_k}$ if 
$w = s_{i_1}s_{i_2} \cdots s_{i_k}$. Here $s_i$ ($1\leqslant i<n$) is the simple transposition interchanging $i$ and 
$i+1$ and $k$ is the number of inversions (or length) of $w$. When $k$ is the length of $w$, the expression
$s_{i_1}s_{i_2} \cdots s_{i_k}$ for $w$ is called a \defn{reduced expression}.
It can be shown that $\partial_w$ is independent of the choice of reduced expression.

There exist degree-preserving divided difference operators $\pi_i$ for $1 \leqslant i < n$, which act on polynomials by
\[
	\pi_i f(x_1,\ldots,x_n) = \partial_i \left( x_i f(x_1,\ldots,x_n) \right).
\]
As with $\partial_i$, we extend this definition to $w\in S_n$, by $\pi_w = \pi_{i_1} \pi_{i_2} \cdots \pi_{i_k}$ if 
$w = s_{i_1}s_{i_2} \cdots s_{i_k}$ is a reduced expression, and $\pi_w$ is independent of the choice of reduced expression.

\begin{definition}
  Given a weak composition $a$ of length $n$, the \defn{Demazure character} $\key_a$ is defined as
  \begin{equation}
    \key_a(x) = \key_a(x_1,\ldots,x_n) = \pi_{w} \left( x_1^{\lambda_1} x_2^{\lambda_2} \cdots x_n^{\lambda_n} \right),
  \end{equation}
  where $\lambda$ is the partition rearrangement of $a$ and $w$ is the shortest permutation that sorts $a$ to $\lambda$.
  \label{def:key}
\end{definition}

For example, we may compute the Demazure character $\key_{(0,2,1,2)}$ by taking $a = (0,2,1,2)$, $\lambda = (2,2,1,0)$ 
and $w = 2431$, and so we have
\begin{eqnarray*}
  \key_{(0,2,1,2)} & = & \pi_1 \pi_3 \pi_2 \pi_3 \left(x_1^2 x_2^2 x_3\right) \\
  & = & \pi_1 \pi_3 \pi_2 \left( x_1^2 x_2^2 x_3 + x_1^2 x_2^2 x_4 \right) \\
  & = & \pi_1 \pi_3 \left( x_1^2 x_2^2 x_3 + x_1^2 x_2^2 x_4 + x_1^2 x_2 x_3^2 + x_1^2 x_2 x_3 x_4 + x_1^2 x_3^2 x_4 
  \right) \\
  & = & \pi_1 \left( x_1^2 x_2^2 x_3 + x_1^2 x_2^2 x_4 + x_1^2 x_2 x_3^2 + 2 x_1^2 x_2 x_3 x_4 + x_1^2 x_2 x_4^2 
  + x_1^2 x_3^2 x_4 + x_1^2 x_3 x_4^2 \right) \\
  & = & x_1^2 x_2^2 x_3 + x_1^2 x_2^2 x_4 + x_1^2 x_2 x_3^2 + 2 x_1^2 x_2 x_3 x_4 + x_1^2 x_2 x_4^2 
  + x_1^2 x_3^2 x_4 + x_1^2 x_3 x_4^2 + x_1 x_2^2 x_3^2 \\
  & & + 2 x_1 x_2^2 x_3 x_4 + x_1 x_2^2 x_4^2 + x_1 x_2 x_3^2 x_4 + x_1 x_2 x_3 x_4^2 + x_2^2 x_3^2 x_4 
  + x_2^2 x_3 x_4^2.
\end{eqnarray*}

Macdonald~\cite{Mac91,Macdonald.1991} showed that when $a$ is weakly increasing of length $n$, we have
\[
	\key_a(x_1,\ldots,x_n) = s_{\mathrm{rev}(a)} (x_1,\ldots,x_n),
\]
where $\mathrm{rev}(a)$ is the partition obtained by reversing (equivalently, sorting) $a$. In particular, Demazure 
characters are a polynomial generalization of irreducible characters.

Making this more precise, Demazure crystals are certain subsets of $B(\lambda)$, which were first conjectured
by Littelmann~\cite{Lit95} to generalize the Demazure characters. This conjecture was later proven
by Kashiwara~\cite{Kashiwara.1993}. Given a subset $X \subseteq B(\lambda)$, we define
$\mathfrak{D}_i$ for $1\leqslant i <n$ as
\begin{equation}
	\mathfrak{D}_i X = \{ b \in B(\lambda) \mid \text{$e_i^k(b) \in X$ for some $k\geqslant 0$}\}.
\end{equation}
For a permutation $w\in S_n$ with reduced expression $w=s_{i_1} s_{i_2} \cdots s_{i_k}$, we define
\begin{equation}
	B_w(\lambda) = \mathfrak{D}_{i_1} \mathfrak{D}_{i_2} \cdots \mathfrak{D}_{i_k} \{ u_\lambda \},
\end{equation}
where $u_\lambda$ is the highest weight element in $B(\lambda)$ satisfying $e_i(u_\lambda)=0$ for all
$1\leqslant i<n$. Whenever $b,b' \in B_w(\lambda) \subseteq B(\lambda)$ and $f_i(b)=b'$ in $B(\lambda)$, then
this crystal operator is also defined in $B_w(\lambda)$.

Let us define the character of a Demazure crystal as
\[
	\operatorname{ch} B_w(\lambda)  = \sum_{b \in B_w(\lambda)} x_1^{\wt(b)_1} \cdots x_n^{\wt(b)_n}.
\]
It was proven by~\cite{Lit95,Kashiwara.1993} that this character coincides with $\key_a$, where $w \cdot a = \lambda$.

%%%%%%%%%%%%%%%%%%%%%%%%%%%%%%%%%%%%%%%%%%%%%%%%%%%%%%%%%%%%%%%%
%
\section{Demazure crystal structure on key tableaux}
%
%%%%%%%%%%%%%%%%%%%%%%%%%%%%%%%%%%%%%%%%%%%%%%%%%%%%%%%%%%%%%%%%
\label{sec:Key}

In Section~\ref{sec:demazure}, we review the combinatorial model of key tableaux \cite{Ass-W} that is central to our results. 
In Section~\ref{sec:SSKT}, we introduce a new crystal structure on semi-standard key tableaux and show that this 
precisely realizes the Demazure character by truncating the crystal structure on semi-standard Young tableaux.

%%%%%%%%%%%%%%%%%%%%%%%%%%%%%%%%%%%%%%%%%%%%%%%%%%%%%%%%%%%%%%%%
%
\subsection{Combinatorics of Demazure characters}
\label{sec:demazure}

Combinatorial interpretations and definitions for Demazure characters for the general linear group were given by Lascoux 
and Sch\"utzenberger~\cite{LS.1990}, Kohnert \cite{Koh91}, Reiner and Shimozono~\cite{ReinerShimozono.1995}, 
and Mason~\cite{Mason.2009}, all of whom refer to them as \defn{key polynomials}. We use an equivalent definition in terms 
of semi-standard key tableaux due to Assaf~\cite{Ass-W}, which is combinatorially equivalent to Mason's semi-skyline augmented fillings but which replaces the triple conditions for more direct row and column conditions (see 
also~\cite{Monical.2016}). Generalizing Young diagrams, given a weak composition $a$, the 
\defn{key diagram of shape $a$} is the array of left-justified cells with $a_i$ boxes in row $i$. 

\begin{definition}[\cite{Ass-W}]
  A \defn{key tableau} is a filling of a key diagram with positive integers such that columns have distinct entries, 
  rows weakly decrease, and, if some entry $i$ is above and in the same column as an entry $k$ with $i<k$, then 
  there is an entry immediately right of $k$, say $j$, with $i<j$.
  \label{def:key-tableau}
\end{definition}

For the Schur polynomial case, we restrict entries in the Young tableaux globally allowing entries $1$ through 
$n$ to appear anywhere. In the Demazure case, we must restrict the entries in the key tableaux locally allowing 
entries to appear only in their row and lower.

\begin{definition}[\cite{Ass-H}]
  A \defn{semi-standard key tableau} is a key tableau in which no entry exceeds its row index. 
  \label{def:SSKT}
\end{definition}

\begin{figure}[ht]
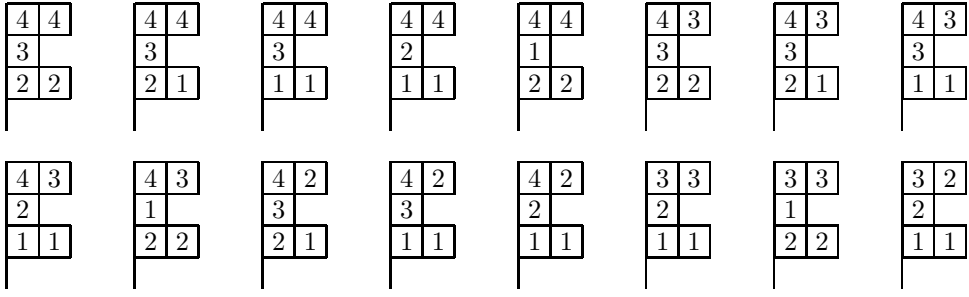

  \begin{displaymath}
    \begin{array}{c@{\hskip 2\cellsize}c@{\hskip 2\cellsize}c@{\hskip 2\cellsize}c@{\hskip 2\cellsize}c@{\hskip 2\cellsize}c@{\hskip 2\cellsize}c@{\hskip 2\cellsize}c}
      \vline\tableau{4 & 4 \\ 3 \\ 2 & 2 \\ & } & 
      \vline\tableau{4 & 4 \\ 3 \\ 2 & 1 \\ & } & 
      \vline\tableau{4 & 4 \\ 3 \\ 1 & 1 \\ & } & 
      \vline\tableau{4 & 4 \\ 2 \\ 1 & 1 \\ & } & 
      \vline\tableau{4 & 4 \\ 1 \\ 2 & 2 \\ & } & 
      \vline\tableau{4 & 3 \\ 3 \\ 2 & 2 \\ & } & 
      \vline\tableau{4 & 3 \\ 3 \\ 2 & 1 \\ & } & 
      \vline\tableau{4 & 3 \\ 3 \\ 1 & 1 \\ & } \\ \\
      \vline\tableau{4 & 3 \\ 2 \\ 1 & 1 \\ & } & 
      \vline\tableau{4 & 3 \\ 1 \\ 2 & 2 \\ & } & 
      \vline\tableau{4 & 2 \\ 3 \\ 2 & 1 \\ & } & 
      \vline\tableau{4 & 2 \\ 3 \\ 1 & 1 \\ & } & 
      \vline\tableau{4 & 2 \\ 2 \\ 1 & 1 \\ & } &       
      \vline\tableau{3 & 3 \\ 2 \\ 1 & 1 \\ & } & 
      \vline\tableau{3 & 3 \\ 1 \\ 2 & 2 \\ & } & 
      \vline\tableau{3 & 2 \\ 2 \\ 1 & 1 \\ & } 
    \end{array}
  \end{displaymath}
  \caption{\label{fig:SSKT}The semi-standard key tableaux of shape $(0,2,1,2)$.}
\end{figure}

For examples, see Figure~\ref{fig:SSKT}. The following property of semi-standard key tableaux will be useful.

\begin{lemma}
  Suppose row $r$ of a semi-standard key tableau has two entries $i+1$ in columns $c$ and $c+1$. If there is 
  an $i$ above row $r$ in column $c$, then there cannot be an $i$ below row $r$ in column $c+1$. 
  \label{lem:cross}
\end{lemma}

\begin{proof}
  If this were the case, then there must be an entry, say $k$, in column $c$ immediately left of the $i$ in column $c+1$. 
  By the weakly decreasing rows condition, $k \geqslant i$, and so by the distinct column entries condition, and $k > i+1$.
  However, since there is an $i+1$ above $k$, the entry immediately right of $k$, which is an $i$, is not larger than $i+1$, 
  a contradiction to the key tableaux column inversion condition.
\end{proof}
  
The \defn{weight} of a semi-standard key tableau $T$, denoted by $\wt(T)$, is the weak composition whose $i$th 
part is the number of occurrences of $i$ in $T$. The following result is proved in \cite{Ass-H} by showing that the semi-standard key tableaux conditions are equivalent to the triple conditions on Mason's semi-skyline augmented fillings \cite{Mason.2009}. This more direct characterization facilitates the constructions to follow.

\begin{theorem}[\cite{Ass-H}]
  The key polynomial $\key_a(x)$ is given by
  \begin{equation}
    \key_{a}(x) = \sum_{T \in \SSKT(a)} x_1^{\wt(T)_1} \cdots x_n^{\wt(T)_n} ,
    \label{e:key}
  \end{equation}
  where $\SSKT(a)$ is the set of semi-standard key tableaux of shape $a$.
  \label{thm:key-SSKT}  
\end{theorem}

The map from standard key tableaux of shape $a$ to standard Young tableaux of shape $\lambda$, 
where $\lambda$ is the unique partition rearrangement of $a$, from \cite{Ass-W} relates the tableaux models 
for key polynomials and Schur polynomials. We extend this map to the semi-standard case as follows.

\begin{definition}
  Given a weak composition $a$ of length $n$, define the \defn{column sorting map} on $\SSKT(a)$ by letting 
  cells fall vertically until there are no gaps between rows, sorting the columns to decrease bottom to top, and then 
  replacing all entries by $i \mapsto n-i+1$.
  \label{def:K2Y}
\end{definition}

For example, the semi-standard key tableaux in Figure~\ref{fig:SSKT} map to the semi-standard Young tableaux 
in the first two rows of Figure~\ref{fig:SSYT}, respectively. The four semi-standard Young tableaux in the bottom 
row of Figure~\ref{fig:SSYT} are not in the image of the column sorting map.

\begin{proposition}
  The column sorting map is a well-defined, injective map $\phi \colon \SSKT(a) \rightarrow \SSYT(\lambda)$, 
  where $\lambda$ is the partition rearrangement of $a$.
  \label{prop:K2Y-well}
\end{proposition}

\begin{proof}
  The column strict condition on semi-standard key tableaux ensures that columns have distinct values. Therefore 
  by construction, a column sorted tableau will have strictly increasing columns. By the column 
  inversion condition for key tableaux, if row $j$ sits above row $i$ and is weakly longer, then column by column 
  the entries in row $j$ must be greater than those in row $i$. Consider applying the column sorting map by first 
  rearranging rows of longest size at the bottom and reversing the relative order of rows of equal length. Since 
  entries within rows are maintained, the weakly decreasing row condition on semi-standard key tableaux is 
  obviously maintained by this process. The column sorting necessarily brings entries from a strictly 
  shorter row down into a longer row. That is, row values can be increased only when the first $k$ values all 
  increase for some $k$, and entries decrease only when the entire row is changed, maintaining the 
  weakly decreasing row condition. Hence the image of the map is indeed a semi-standard Young tableau 
  of shape $\lambda$.

  To see that the map is injective, we can define an inverse map by first applying $i \mapsto n-i+1$ to all letters
  in a semi-standard Young tableau. Then fill the shape of $a$ column by column 
  from right to left, and within a column from bottom to top, according to the columns of the given tableau
  by selecting at each step the smallest available entry that maintains the weakly decreasing 
  row condition. To see that the column inversion condition for key tableaux still holds, suppose $j$ is the smallest 
  label available that can be placed in cell $C$ in order to satisfy weakly decreasing rows. It is easy to see that 
  the column inversion condition for key tableaux is maintained, but it could happen that an entry is placed in a 
  row with smaller index. The tableaux for which this occurs are precisely the ones not in the image of the column 
  sorting map.  
\end{proof}

%%%%%%%%%%%%%%%%%%%%%%%%%%%%%%%%%%%%%%%%%%%%%%%%%%%%%%%%%%%%%%%%
%
\subsection{Crystal operators on semi-standard key tableaux}
\label{sec:SSKT}

We generalize the crystal structure on semi-standard Young tableaux to a Demazure crystal structure on 
semi-standard key tableaux as follows.

For a word $w$ of length $k$ with letters in the alphabet $A=\{1,2,\ldots,n\}$, an integer $1 \leqslant r \leqslant  k$, 
and an integer $1\leqslant i<n$, define
\begin{equation}
  m_i(w,r) = \wt(w_{r} w_{r+1} \cdots w_{k})_{i+1} - \wt(w_{r} w_{r+1} \cdots w_{k})_{i}.
\end{equation}
Set $m_i(w) = \max_r\{m_i(w,r)\}$. Observe that if $m_i(w) > 0$ and $q$ is the rightmost occurrence of this maximum, 
then $w_q = i+1$, and if $p$ is the leftmost occurrence of this maximum, then either $p=1$ or $w_{p-1} = i$.

For $T$ a key tableau, the \defn{column reading word of $T$}, denoted by $w(T)$, is the word obtained by reading 
the entries of $T$ down columns from right to left. Note that columns for key tableaux are read in the reverse order 
as columns for Young tableaux. For example, the column reading word of the leftmost key tableau in the top row of
 Figure~\ref{fig:SSKT} is $42432$.

\begin{definition}
  Given an integer $1\leqslant i<n$, define the \defn{raising operators} $e_i$ on semi-standard key tableaux 
  of shape $a$ of length $n$ as follows: if $m_i(w(T)) \leqslant 0$, then $e_i(T)=0$; otherwise, let $q$ be the largest 
  index such that $m_i(w(T),q) = m_i(w(T))$, and $e_i(T)$ changes all entries $i+1$ weakly right of the entry in $T$ 
  corresponding to $w(T)_q$ to $i$ and change all $i$'s in the same columns as these entries to $i+1$'s.
  \label{def:key-raise}
\end{definition}

For an example of the raising operator $e_1$, see Figure~\ref{fig:key-raise}. For this example, the column reading 
word is given below each key tableau with the largest index that attains $m_1(w(T))>0$ underlined and the 
corresponding entry in the tableau circled.

\begin{figure}[ht]
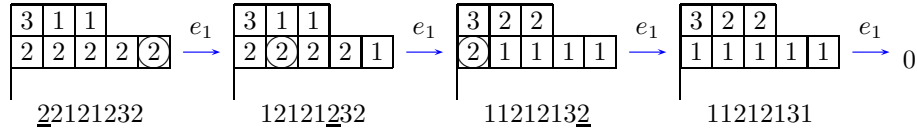

  \begin{displaymath}
    \begin{array}{c@{\hskip 2\cellsize}c@{\hskip 2\cellsize}c@{\hskip 2\cellsize}c@{\hskip 2\cellsize}c}
      \rnode{A}{\vline\tableau{3 & 1 & 1 \\ 2 & 2 & 2 & 2 & \cir{2} \\ & }} &
      \rnode{B}{\vline\tableau{3 & 1 & 1 \\ 2 & \cir{2} & 2 & 2 & 1 \\ & }} &
      \rnode{C}{\vline\tableau{3 & 2 & 2 \\ \cir{2} & 1 & 1 & 1 & 1 \\ & }} &
      \rnode{D}{\vline\tableau{3 & 2 & 2 \\ 1 & 1 & 1 & 1 & 1       \\ & }} &
      \rnode{E}{\raisebox{-\cellsize}{$0$}} \\
      \underline{2}2121232 &
      12121\underline{2}32 &
      1121213\underline{2} &       
      11212131 &       
    \end{array}
    \psset{nodesep=5pt,linewidth=.1ex}
    \ncline[linecolor=blue]{->} {A}{B} \naput{e_1}
    \ncline[linecolor=blue]{->} {B}{C} \naput{e_1}
    \ncline[linecolor=blue]{->} {C}{D} \naput{e_1}
    \ncline[linecolor=blue]{->} {D}{E} \naput{e_1}
  \end{displaymath}
  \caption{\label{fig:key-raise}An example of the raising operator $e_1$ on key tableaux.}
\end{figure}

\begin{proposition}
  The raising operator $e_i \colon \SSKT(a) \rightarrow \SSKT(a) \cup \{0\}$ is a well-defined map. Moreover, the 
  restriction of $e_i$ to the pre-image $e_i^{-1}(\SSKT(a))$ satisfies $\wt(e_i(T))_i = \wt(T)_i +1$, 
  $\wt(e_i(T))_{i+1} = \wt(T)_{i+1}-1$, and $\wt(e_i(T))_j = \wt(T)_j$ for all $j \neq i,i+1$.
  \label{prop:key-well}
\end{proposition}

\begin{proof}
  Let $T \in \SSKT(a)$, set $m = m_i(w(T))$, and suppose $m>0$. Let $x$, say in row $r$ and column $c$, 
  be the cell in $T$ that attains $m$ at the rightmost position in column reading order. We claim that every cell weakly 
  right of $x$ in row $r$ with entry $i+1$ except for one has an $i$ above it. If the entry immediately right of $x$ is $h$ 
  for some $h < i+1$, then the key tableaux conditions ensure that there cannot be an $i$ above $x$ since $h \leqslant i$. 
  Suppose, then, that there is an $i+1$ immediately right of $x$. Since $x$ attains the maximum $m$ and there is 
  an $i+1$ to its left, we must have an $i$ between them in column reading order. Thus there must be an 
  $i$ either below row $r$ in column $c+1$ or above row $r$ in column $c$. If there is an $i$ in row $r^{\prime}<r$ 
  in column $c+1$, then there must be an entry, say $k$, in row $r^{\prime}$ in column $c$ satisfying $k \geqslant i$. 
  Moreover, by the key tableau column inversion condition, we cannot have $k>i+1$ since $i+1>i$. Therefore $k=i$, 
  in which case $x$ cannot be the rightmost position to attain $m$, a contradiction. Moreover, it now follows by 
  induction from Lemma~\ref{lem:cross} that every $i+1$ right of $x$ in row $r$ either has an $i$ above it or no $i$ 
  in its column, and the latter cannot be the case more than once else the rightmost $i+1$ would have $i$-index 
  greater than $m$. This proves the claim, from which it follows that one more cell changes entry from $i+1$ to $i$ 
  than the reverse, thus proving $\wt(e_i(T))_i = \wt(T)_i +1$, $\wt(e_i(T))_{i+1} = \wt(T)_{i+1}-1$, and 
  $\wt(e_i(T))_j = \wt(T)_j$ for all $j \neq i,i+1$. 

  Next we show that rows of $e_i(T)$ are weakly decreasing. This is clear for row $r$ since all $i+1$ weakly 
  right of $x$ change to $i$. If $i$ changes to $i+1$ in cell $y$ and the cell immediately left of $y$ also contains an $i$, 
  then this $i$ also is changed to an $i+1$. This is clear from the previous analysis provided $y$
  is not in the column of $x$; if $y$ is in the column of $x$ and has an $i$ immediately to its left, then $x$ cannot 
  be the rightmost cell in column reading order to attain $m$. 

  Next we show that columns of $e_i(T)$ have distinct entries. Since $x$ cannot have an $i$ below it and be the 
  leftmost cell in column reading order to attain $m$, any $i+1$ that changes to an $i$ either has no $i$ in the column 
  or an $i$ above it. In the latter case, this $i$ will become an $i+1$. 

  Next we show that $e_i(T)$ if $a<c$ with $a$ above $c$, then there is an entry $b$ immediately right of $c$ with $a<b$. If a column contains $i$ and not $i+1$, then nothing is changed, and if it has both, then the $i+1$ appears above $i$ in $e_i(T)$. Therefore the only potential problem occurs when $b=i+1$ in $T$ is changed to $i$ in $e_i(T)$ and $a=i$. In this case, if the column of $a$ has no $i+1$, then $b$ does not attain $m$ and is not changed to $i$, and otherwise both $a$ and $c$ change removing the inversion triple from consideration. 

  Finally, decrementing values maintains the property that entries do not exceed their row index, and $i$ changes to 
  $i+1$ only when it sits above an $i+1$, so these entries lie strictly above row $i+1$. Therefore $e_i(T)$ is a 
  semi-standard key tableau.
\end{proof}

\begin{figure}[ht]
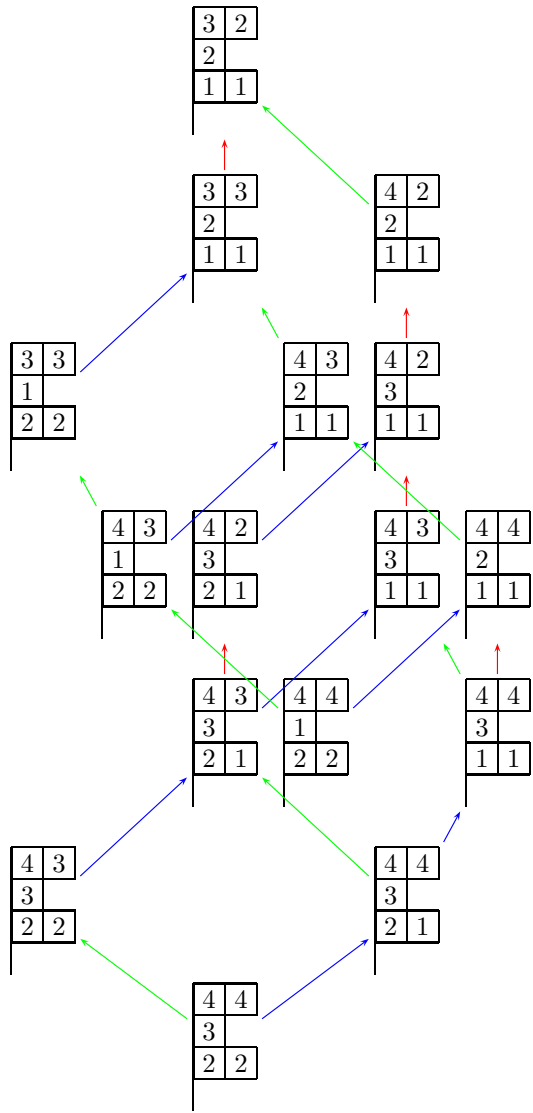

  \begin{displaymath}
    \begin{array}{cccccc}
      & & \rnode{b1}{\vline\tableau{3 & 2 \\ 2 \\ 1 & 1 \\ & }} & & & \\[4\cellsize]
      & & \rnode{b2}{\vline\tableau{3 & 3 \\ 2 \\ 1 & 1 \\ & }} & & \rnode{c2}{\vline\tableau{4 & 2 \\ 2 \\ 1 & 1 \\ & }} & \\[4\cellsize]
      \rnode{a3}{\vline\tableau{3 & 3 \\ 1 \\ 2 & 2 \\ & }} & & & \rnode{c3}{\vline\tableau{4 & 3 \\ 2 \\ 1 & 1 \\ & }} & \rnode{C3}{\vline\tableau{4 & 2 \\ 3 \\ 1 & 1 \\ & }} & \\[4\cellsize]
      & \rnode{b4}{\vline\tableau{4 & 3 \\ 1 \\ 2 & 2 \\ & }} & \rnode{B4}{\vline\tableau{4 & 2 \\ 3 \\ 2 & 1 \\ & }} & & \rnode{c4}{\vline\tableau{4 & 3 \\ 3 \\ 1 & 1 \\ & }} & \rnode{d4}{\vline\tableau{4 & 4 \\ 2 \\ 1 & 1 \\ & }} \\[4\cellsize]
      & & \rnode{b5}{\vline\tableau{4 & 3 \\ 3 \\ 2 & 1 \\ & }} & \rnode{c5}{\vline\tableau{4 & 4 \\ 1 \\ 2 & 2 \\ & }} & & \rnode{d5}{\vline\tableau{4 & 4 \\ 3 \\ 1 & 1 \\ & }} \\[4\cellsize]
      \rnode{a6}{\vline\tableau{4 & 3 \\ 3 \\ 2 & 2 \\ & }} & & & & \rnode{c6}{\vline\tableau{4 & 4 \\ 3 \\ 2 & 1 \\ & }} & \\[3\cellsize]
      & & \rnode{b7}{\vline\tableau{4 & 4 \\ 3 \\ 2 & 2 \\ & }} & & & 
    \end{array}
    \psset{nodesep=2pt,linewidth=.1ex}
    \ncline[linecolor=red]{<-}  {b1}{b2} %\ncput*{2}
    \ncline[linecolor=green]{<-}{b1}{c2} %\ncput*{3}
    \ncline[linecolor=blue]{<-} {b2}{a3} %\ncput*{1}
    \ncline[linecolor=red]{<-}  {c2}{C3} %\ncput*{2}
    \ncline[linecolor=green]{<-}{b2}{c3} %\ncput*{3}
    \ncline[linecolor=blue]{<-} {c3}{b4} %\ncput*{1}
    \ncline[linecolor=blue]{<-} {C3}{B4} %\ncput*{1}
    \ncline[linecolor=red]{<-}  {C3}{c4} %\ncput*{2}
    \ncline[linecolor=green]{<-}{a3}{b4} %\ncput*{3}
    \ncline[linecolor=green]{<-}{c3}{d4} %\ncput*{3}
    \ncline[linecolor=blue]{<-} {c4}{b5} %\ncput*{1}
    \ncline[linecolor=blue]{<-} {d4}{c5} %\ncput*{1}
    \ncline[linecolor=red]{<-}  {B4}{b5} %\ncput*{2}
    \ncline[linecolor=red]{<-}  {d4}{d5} %\ncput*{2}
    \ncline[linecolor=green]{<-}{b4}{c5} %\ncput*{3}
    \ncline[linecolor=green]{<-}{c4}{d5} %\ncput*{3}
    \ncline[linecolor=blue]{<-} {b5}{a6} %\ncput*{1}
    \ncline[linecolor=blue]{<-} {d5}{c6} %\ncput*{1}
    \ncline[linecolor=green]{<-}{b5}{c6} %\ncput*{3}
    \ncline[linecolor=blue]{<-} {c6}{b7} %\ncput*{1}
    \ncline[linecolor=green]{<-}{a6}{b7} %\ncput*{3}
  \end{displaymath}
  \caption{\label{fig:Key-0212}The crystal structure on $\SSKT(0,2,1,2)$, with edges $e_1 \color{blue}\nearrow$, $e_2 \color{red}\uparrow$, $e_3 \color{green}\nwarrow$.}
\end{figure}

\begin{lemma}
  For $T \in \SSKT(a)$ and for any $1 \leqslant i <n$, $e_i(T) \neq 0$ if and only if $f_{n-i}(\phi(T))\neq 0$. 
  In this case, we have $\phi(e_i(T)) = f_{n-i}(\phi(T))$, where $\phi$ denotes the column sorting map. 
  \label{lem:commute}
\end{lemma}

\begin{proof}
  Given a word $w = w_1 w_2 \cdots w_k$ with $1\leqslant w_j \leqslant n$ for all $j$, let 
  $u = (n-w_k+1) (n-w_{k-1}+1) \cdots (n-w_1+1)$. 
  Then $m_i(w,r) = M_{n-i}(u,k-r+1)$, and $q$ is the index of the rightmost occurrence of $m_i(w)$ in $w$ if and only if 
  $k-q+1$ is the index of the leftmost occurrence of $M_{n-i}(u)$ in $u$. If $T \in \SSKT(a)$ has no column inversions, 
  then since the column reading word of a semi-standard key tableau is right to left and the column reading word 
  of a semi-standard Young tableau is left to right, $w(T)$ and $w(\phi(T))$ precisely have the relationship of 
  $w$ and $u$, and the result follows.

  In the general case, since $e_i$ and $f_i$ depend only on the letters $i,i+1$, we may restrict our attention to the 
  subword on those letters. In doing so, notice that columns with $i$ above $i+1$ appear in consecutive runs separated 
  at least by a column immediately right of the run with an $i+1$ and no $i$. In the column reading word, this manifests itself
  as a string of alternating $i$'s and $i+1$'s that begins and ends with an $i+1$. If we let $q^{\prime}$ denote the leftmost 
  $i+1$ in the alternating string that attains $m_i(w(T))$, then $k-q^{\prime}+1$ is the smallest index that attains 
  $M_{n-i}(w(\phi(T)))$. That is, the rightmost column of $T$ in which an $i+1$ changes to an $i$ \emph{without} an 
  $i$ also changing to an $i+1$ in passing to $e_i(T)$ is precisely the column of $\phi(T)$ in which an $n-i$ changes 
  to an $n-i+1$ in passing to $f_{n-i}(\phi(T))$.
\end{proof}

For example, the semi-standard key tableaux of shape $(0,5,3)$ in Figure~\ref{fig:key-raise} map by the column sorting 
map to the semi-standard Young tableaux of shape $(5,3)$ in Figure~\ref{fig:young-lower}, and the raising operator 
$e_1$ on the former becomes the lowering operator $f_2$ on the latter.

\begin{definition}
  Given an integer $1\leqslant i<n$, define the \defn{lowering operator} $f_i$ on semi-standard key tableaux 
  of shape $a$ as follows: 
  let $p$ be the smallest index such that $m_i(w(T),p) = m_i(w(T))$. If $p=1$ or if the entry in $T$ corresponding to 
  $w_p$ lies in row $i$, then $f_i(T) = 0$; otherwise $f_i(T)$ changes all entries $i$ weakly right of the entry 
  in $T$ corresponding to $w_{p-1}$ to $i+1$ and change all $i$'s in the same columns as these entries to $i$'s.
  \label{def:key-lower}
\end{definition}

For examples of lowering operators on semi-standard key tableaux, see Figure~\ref{fig:Key-0212}
($f_i$ are inverses of $e_i$ when they are defined on an element).

\begin{proposition}
\label{proposition.f intertwine key}
  For $T\in\SSKT(a)$ and for any $1 \leqslant i <n$, if there exists $S\in\SSKT(a)$ such that $e_i(S)=T$, then 
  $f_i(T)=S$, and otherwise $f_i(T)=0$. In particular, the lowering operator $f_i$ is well-defined and if $f_i(T)\neq 0$, 
  then it satisfies $\wt(f_i(T))_i = \wt(T)_i +1$, $\wt(f_i(T))_{i+1} = \wt(T)_{i+1}-1$, and $\wt(f_i(T))_j = \wt(T)_j$ for 
  all $j \neq i,i+1$. Moreover, letting $\phi$ denote the column sorting map, if $f_i(T)\neq 0$, then we have 
  $\phi(f_i(T)) = e_{n-i}(\phi(T))$.
\end{proposition}

\begin{proof}
  Recall from the analysis in the proof of Proposition~\ref{prop:key-well} that when $e_i(S) \neq 0$, $w(S)$ and 
  $w(e_i(S))$ differ on the restriction to letters $i,i+1$ precisely in that an alternating string beginning and ending 
  with $i+1$ for which the last entry is the rightmost to attain $m_i(w(S))$ becomes an alternating string beginning 
  and ending with $i$ for which the first entry is immediately left of the leftmost to attain $m_i(w(e_i(S)))$. Therefore 
  if $e_i(S)=T$, then $f_i(T)=S$. We have $f_i(T)=0$ precisely when there is no place to act 
  (when $p=1$) or when acting would violate the semi-standard key tableaux condition that entries cannot exceed 
  their row index. The remainder of the result follows from Proposition~\ref{prop:key-well} and Lemma~\ref{lem:commute}.
\end{proof}

%%%%%%%%%%%%%%%%%%%%%%%%%%%%%%%%%%%%%%%%%%%%%%%%%%%%%%%%%%%%%%%%
%
\subsection{Demazure crystal on semi-standard key tableaux}
\label{sec:dem}

To arrive at our main result, that the raising and lowering operators on semi-standard key tableaux give a Demazure crystal, 
we refine the column sorting map to an injective map between semi-standard key tableaux for different weak compositions.

\begin{lemma}
  Given a weak composition $a$ and an index $i$ such that $a_i < a_{i+1}$, for $T \in \SSKT(a)$ such that $e_i(T)=0$, 
  there exists $S \in \SSKT(s_i a)$ such that $\phi(T) = \phi(S)$, where $\phi$ is the column sorting map.
  \label{lem:CSM}
\end{lemma}

\begin{proof}
  The statement is equivalent to the assertion that there exists $S \in \SSKT(s_i a)$ with the same column sets as $T$. 
  We may describe the map from T to S explicitly as follows. First, move the $a_{i+1}-a_i$ rightmost cells in row $i+1$ 
  down to row $i$. Since $e_i(T)=0$, there cannot be a letter $i+1$ that is moved down since if any of these cells contain
  an $i+1$, there will be a positive index allowing $e_i$ to act non-trivially. If, after this, row $i$ is not weakly decreasing, 
  then swap the entries in rows $i$ and $i+1$ of the offending column. Since $e_i(T)=0$, there cannot be any letters $i+1$
  that are moved down at this step either, so the resulting tableau $S$ has no entry exceeding its row index. Rows clearly
  maintain their weakly decreasing status, and it is easy to see that no violations of the column inversion condition can 
  arise. Therefore $S \in \SSKT(s_i a)$.
\end{proof}

Lemma~\ref{lem:CSM} ensures that the following operators are well-defined on semi-standard key tableaux.

\begin{definition}
  Given a weak composition $a$ and an index $i$ such that $a_i < a_{i+1}$, define an operator $\mathcal{E}_i$ on 
  $\SSKT(a)$ by $\mathcal{E}_i(T) = S$, where $S\in\SSKT(s_i a)$ satisfies $\phi(S) = \phi(e_i^{k-1}(T))$ for $k$ minimal 
  such that $e_i^k(T)=0$.
  \label{def:pi-key}
\end{definition}

For examples of $\mathcal{E}_i$, see Figure~\ref{fig:pi-tab}. Similar to $\pi_w$ and $\partial_w$, we may extend this to 
define $\mathcal{E}_w=\mathcal{E}_{i_1} \cdots \mathcal{E}_{i_k}$, where $s_{i_1} \cdots s_{i_k}$ 
is any reduced expression for $w$. It is easy to see that this is well-defined from the local relations of the type $A$ crystal 
operators on tableaux as characterized by Stembridge \cite{Ste03}. 

\begin{figure}[ht]
  \begin{displaymath}
    \begin{array}{c@{\hskip3\cellsize}c@{\hskip3\cellsize}c@{\hskip3\cellsize}c@{\hskip3\cellsize}c}
      & \rnode{b}{\vline\tableau{1 \\ 3 & 3 \\ 2 & 2 \\ & }} & & & \\ \\
      \rnode{A}{\vline\tableau{4 & 3 \\ 1 \\ 2 & 2 \\ & }} &
      \rnode{B}{\vline\tableau{4 & 3 \\ 2  \\ & \\ 1 & 1}} &
      \rnode{C}{\vline\tableau{2 \\ 3 & 3 \\ & \\ 1 & 1 }} &
      \rnode{D}{\vline\tableau{3 \\ & \\ 2 & 2 \\ 1 & 1 }} &
      \rnode{E}{\vline\tableau{ & \\ 3 \\ 2 & 2 \\ 1 & 1}} \\ \\
      & &
      \rnode{c}{\vline\tableau{4 & 3 \\ & \\ 2 \\ 1 & 1 }} &
      \rnode{d}{\vline\tableau{ & \\ 3 & 3 \\ 2 \\ 1 & 1 }} &
    \end{array}
    \psset{nodesep=2pt,linewidth=.1ex}
    \ncline[linecolor=green]{->} {A}{b} \naput{\mathcal{E}_{3}}    
    \ncline[linecolor=blue ]{->} {b}{C} \naput{\mathcal{E}_{1}}
    \ncline[linecolor=blue ]{->} {A}{B} \naput{\mathcal{E}_{1}}    
    \ncline[linecolor=green]{->} {B}{C} \naput{\mathcal{E}_{3}}
    \ncline[linecolor=red  ]{->} {C}{D} \naput{\mathcal{E}_{2}}
    \ncline[linecolor=green]{->} {D}{E} \naput{\mathcal{E}_{3}}
    \ncline[linecolor=red  ]{->} {B}{c} \nbput{\mathcal{E}_{2}}
    \ncline[linecolor=green]{->} {c}{d} \naput{\mathcal{E}_{3}}
    \ncline[linecolor=red  ]{->} {d}{E} \nbput{\mathcal{E}_{2}}
  \end{displaymath}
  \caption{\label{fig:pi-tab}An example of the $\mathcal{E}_i$ operators on semi-standard key tableaux.}
\end{figure}

Given a weak composition $a$, for $w$ the permutation that sorts $a$ to partition shape $\lambda$, the operator 
$\mathcal{E}_w$ takes $T\in\SSKT(a)$ to the highest weight element of the crystal along edges specified by $w$. 
This is precisely the statement needed to show that the crystal operators defined on semi-standard key tableaux of 
shape $a$ realize the Demazure crystal for $w$.

\begin{theorem}
\label{theorem.key demazure}
  Let $a$ be a weak composition that sorts to the partition $\lambda$. The raising and lowering operators on $\SSKT(a)$ 
  give the Demazure crystal for highest weight $\lambda$ truncating with respect to the minimal length permutation $w$ that 
  sorts $a$ to $\lambda$.
\end{theorem}

\begin{proof}
  Given $T \in \SSKT(a)$, for $w$ the permutation that sorts $a$ to partition shape $\lambda$, we necessarily 
  have $\mathcal{E}_w(T) \in \SSKT(\lambda)$. However, the constraint that entries cannot exceed their row index 
  together with distinct column values forces $\SSKT(\lambda)$ to have a single element, the tableau with all entries 
  in row $i$ equal to $i$. In particular, this element maps via the column sorting map to the highest weight $u_{\lambda}$. 
  By Lemma~\ref{lem:commute}, this means $T \in \mathfrak{D}_w \{ u_{\lambda} \}$ for every $T \in \SSKT(a)$, and so 
  $\phi(\SSKT(a)) \subseteq B_w(\lambda)$. By Theorem~\ref{thm:key-SSKT}, the sums of the weights on both sides 
  agree, so we must have equality.
\end{proof}

For example, removing the four vertices of the $(2,2,1)$-crystal in Figure~\ref{fig:Young-221} corresponding to the 
four semi-standard Young tableaux of shape $(2,2,1)$ that are not in the image of the column sorting map on 
semi-standard key tableaux of shape $(0,2,1,2)$ precisely gives the $(0,2,1,2)$-Demazure crystal in Figure~\ref{fig:Key-0212}.

%%%%%%%%%%%%%%%%%%%%%%%%%%%%%%%%%%%%%%%%%%%%%%%%%%%%%%%%%%%%%%%%
%
\section{Crystal structure for Stanley symmetric polynomials}
%
%%%%%%%%%%%%%%%%%%%%%%%%%%%%%%%%%%%%%%%%%%%%%%%%%%%%%%%%%%%%%%%%
\label{sec:stanley}

We review the combinatorics of Stanley symmetric functions and polynomials in terms of reduced factorizations of 
a permutation in Section~\ref{sec:RF}. We proceed in Section~\ref{sec:EG} to review Edelman--Greene insertion
and review the crystal structure on reduced factorizations as recently introduced in~\cite{MS16} in
Section~\ref{sec:RF-crystal}. 

%%%%%%%%%%%%%%%%%%%%%%%%%%%%%%%%%%%%%%%%%%%%%%%%%%%%%%%%%%%%%%%%
%
\subsection{Combinatorics of Stanley symmetric functions}
\label{sec:RF}

Stanley \cite{Stanley.1984} introduced a new family of symmetric functions to enumerate reduced expressions for permutations.

\begin{definition}
  A \defn{reduced word} for a permutation $w\in S_n$ is a word $i_1 \ldots i_k$ such that $s_{i_1} \cdots s_{i_k} = w$ where 
  $k$ is the inversion number of $w$.
  \label{def:reduced}
\end{definition}

For example, there are $11$ reduced words for the permutation $153264$ as shown in Figure~\ref{fig:reduced}.

\begin{figure}[ht]
  \begin{displaymath}
    \begin{array}{ccccccccccc}
      45323 & 45232 & 43523 & 42532 & 43253 & 24532 & 42352 & 43235 & 24352 & 42325 & 24325
    \end{array}
  \end{displaymath}
   \caption{\label{fig:reduced}The reduced words for $153264$.}
\end{figure}

\begin{definition}
  Given a reduced word $\rho$, an \defn{increasing factorization} for $\rho$ partitions the word $\rho$ into 
  (possibly empty) blocks (or factors) such that entries increase left to right within each block. 
  
  Given a permutation $w$, a \defn{reduced factorization} for $w$ is an increasing factorization of a reduced word for $w$. 
  Denote the set of reduced factorizations for $w$ by $\RF(w)$. 
   \label{def:factor}
\end{definition}

For example, the reduced factorizations for $153264$ into $4$ blocks are shown in Figure~\ref{fig:factor}.

\begin{figure}[ht]
  \begin{displaymath}
    \begin{array}{cccccc}
      ()(45)(3)(23) & ()(45)(23)(2) & ()(4)(35)(23) & (4)(25)(3)(2) & (4)(3)(25)(3) & ()(245)(3)(2) \\ 
      (4)(5)(3)(23) & (4)(5)(23)(2) & (4)()(35)(23) &               &               & (2)(45)(3)(2) \\ 
      (45)()(3)(23) & (45)()(23)(2) & (4)(3)(5)(23) &               &               & (24)(5)(3)(2) \\ 
      (45)(3)()(23) & (45)(2)(3)(2) & (4)(35)()(23) &               &               & (245)()(3)(2) \\ 
      (45)(3)(2)(3) & (45)(23)()(2) & (4)(35)(2)(3) &               &               & (245)(3)()(2) \\ 
      (45)(3)(23)() & (45)(23)(2)() & (4)(35)(23)() &               &               & (245)(3)(2)() \\ \\
      ()(4)(235)(2) & ()(4)(3)(235) & ()(24)(35)(2) & ()(4)(23)(25) & ()(24)(3)(25) & \\ 
      (4)()(235)(2) & (4)()(3)(235) & (2)(4)(35)(2) & (4)()(23)(25) & (2)(4)(3)(25) & \\ 
      (4)(2)(35)(2) & (4)(3)()(235) & (24)()(35)(2) & (4)(2)(3)(25) & (24)()(3)(25) & \\ 
      (4)(23)(5)(2) & (4)(3)(2)(35) & (24)(3)(5)(2) & (4)(23)()(25) & (24)(3)()(25) & \\ 
      (4)(235)()(2) & (4)(3)(23)(5) & (24)(35)()(2) & (4)(23)(2)(5) & (24)(3)(2)(5) & \\ 
      (4)(235)(2)() & (4)(3)(235)() & (24)(35)(2)() & (4)(23)(25)() & (24)(3)(25)() &    
    \end{array}
  \end{displaymath}
   \caption{\label{fig:factor}The reduced factorizations for $153264$ into $4$ blocks.}
\end{figure}

The \defn{weight} of a reduced factorization $r$, denoted by $\wt(r)$, is the weak composition whose $i$th part is 
the number of letters in the $i$th block of $r$ from the right. For example, $\wt((45)(3)(23)()) = (0,2,1,2)$.

\begin{definition}
  The \defn{Stanley symmetric function} indexed by the permutation $w$ is
  \begin{equation}
    F_w(x) = \sum_{r\in \RF(w^{-1})} x^{\wt(r)}.
  \end{equation}
  \label{def:stanley}
\end{definition}

Therefore we compute $F_{143625}$ using reduced factorizations for $143625^{-1} = 153264$.

Note that reduced factorizations can, in principle, have an arbitrary number of blocks and hence $F_w(x)$ is 
a symmetric function in infinitely many variables $x=(x_1,x_2,\ldots)$.

We can restrict Stanley symmetric functions to Stanley symmetric polynomials by restricting the number of blocks
 in the reduced factorizations. Let $\RF^\ell(w)$ be the set of reduced factorizations of $w$ with precisely $\ell$ blocks. 
 Then the \defn{Stanley symmetric polynomial} in $\ell$ variables is
\[
	F_w(x_1,x_2,\ldots,x_\ell) = \sum_{r\in \RF^\ell(w^{-1})} x^{\wt(r)}.
\]

%%%%%%%%%%%%%%%%%%%%%%%%%%%%%%%%%%%%%%%%%%%%%%%%%%%%%%%%%%%%%%%%
%
\subsection{Edelman--Greene correspondence}
\label{sec:EG}

In their study of Stanley symmetric functions, Edelman and Greene \cite{EG.1987} developed the following insertion 
algorithm that they used to give a formula for the Schur expansion of Stanley symmetric functions.

\begin{definition}\cite[Definition~6.21]{EG.1987}
  Let $P$ be a Young tableau, and let $x$ be a positive integer. Let $P_i$ be the $i$th lowest row of $P$. Define 
  the \defn{Edelman-Greene insertion of $x$ into $P$}, denoted by $P \leftarrow x$, as follows. Set 
  $x_0=x$ and for $i\geqslant 0$, insert $x_i$ into $P_{i+1}$ as follows. If $x_i \geqslant z$ for all $z\in P_{i+1}$, place $x_i$ 
  at the end of $P_{i+1}$ and stop. Otherwise, let $x_{i+1}$ denote the smallest element of $P_{i+1}$ such that 
  $x_{i+1}>x_i$. If $x_{i+1} \neq x_i+1$ or $x_i$ is not already in $P_{i+1}$, replace $x_{i+1}$ by $x_i$ in $P_{i+1}$ 
  and continue (we say that $x_i$ \defn{bumps} $x_{i+1}$ in row $i+1$). Otherwise leave $P_{i+1}$ unchanged and 
  continue with $x_{i+1}$.
  \label{def:insert-EG}
\end{definition}

Given a reduced expression $\rho$, define the \defn{insertion tableau} for $\rho$, denoted by $P(\rho)$, to be the 
result of inserting the word for $\rho$ letter by letter into the empty tableau. To track the growth of $P(\rho)$, define 
the \defn{recording tableau} for $\rho$, denoted by $Q(\rho)$, to be the result of adding $i$ into the new cell created 
when inserting the $i$th letter. For example, Figure~\ref{fig:EG} constructs the insertion tableau (top) and recording 
tableau (bottom) for the reduced expression $45232$.

\begin{figure}[ht]
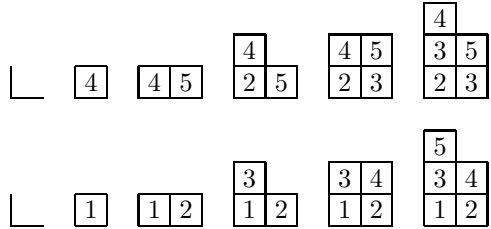

  \begin{displaymath}
    \begin{array}{c@{\hskip\cellsize}c@{\hskip\cellsize}c@{\hskip\cellsize}c@{\hskip\cellsize}c@{\hskip\cellsize}c}
      \raisebox{-2\cellsize}{$\vline\tableau{\\\hline}$} &
      \tableau{\\ \\ 4} &
      \tableau{\\ \\ 4 & 5} &
      \tableau{\\ 4 \\ 2 & 5} &
      \tableau{\\ 4 & 5 \\ 2 & 3} &
      \tableau{4 \\ 3 & 5 \\ 2 & 3} \\ \\
      \raisebox{-2\cellsize}{$\vline\tableau{\\\hline}$} &
      \tableau{\\ \\ 1} &
      \tableau{\\ \\ 1 & 2} &
      \tableau{\\ 3 \\ 1 & 2} &
      \tableau{\\ 3 & 4 \\ 1 & 2} &
      \tableau{5 \\ 3 & 4 \\ 1 & 2}
    \end{array}
  \end{displaymath}
  \caption{\label{fig:EG}The insertion and recording tableaux for the reduced expression $45232$.}
\end{figure}

\begin{theorem}\cite[Theorem~6.25]{EG.1987}
  The Edelman--Greene correspondence $\rho \mapsto (P(\rho),Q(\rho))$ is a bijection between reduced 
  expressions and all pairs of tableaux $(P,Q)$ such that $P$ and $Q$ have the same shape, $P$ is increasing 
  with $\mathrm{row}(P)$ a reduced word, and $Q$ is standard.
  \label{thm:insert-EG}
\end{theorem}

We may extend the Edelman--Greene correspondence to a bijection between reduced factorizations and all pairs 
of tableaux $(P,Q)$ such that $P$ and $Q$ have the same shape, $P$ is increasing with $\mathrm{row}(P)$ a 
reduced word, and $Q$ is semi-standard. To do so, given a reduced factorization $r$ into $\ell$ blocks, define 
$P(r)$ to be $P(\rho)$ where $\rho$ is the underlying reduced expression for $r$, and define $Q(r)$ to be the 
result of adding $\ell-i+1$ into each new cell created when inserting a letter from block $i$ (from the right). For example, the 
recording tableau for the reduced factorization $(4)(5)(23)(2)$ is constructed in Figure~\ref{fig:EG-factor}.

\begin{figure}[ht]
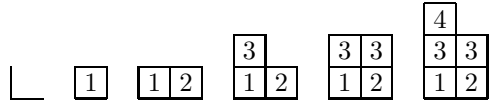

  \begin{displaymath}
    \begin{array}{c@{\hskip\cellsize}c@{\hskip\cellsize}c@{\hskip\cellsize}c@{\hskip\cellsize}c@{\hskip\cellsize}c}
      \raisebox{-2\cellsize}{$\vline\tableau{\\\hline}$} &
      \tableau{\\ \\ 1} &
      \tableau{\\ \\ 1 & 2} &
      \tableau{\\ 3 \\ 1 & 2} &
      \tableau{\\ 3 & 3 \\ 1 & 2} &
      \tableau{4 \\ 3 & 3 \\ 1 & 2}
    \end{array}
  \end{displaymath}
  \caption{\label{fig:EG-factor}The recording tableau for the reduced factorization $(4)(5)(23)(2)$.}
\end{figure}

\begin{corollary}
  The correspondence $r \mapsto (P(r),Q(r))$ is a bijection between reduced factorizations and all pairs of 
  tableaux $(P,Q)$ such that $P$ and $Q$ have the same shape, $P$ is increasing with $\mathrm{row}(P)$ a 
  reduced word, and $Q$ is semi-standard. Moreover, if $r$ has $\ell$ blocks, then $\wt(Q(r))_i = \wt(r)_{\ell-i}$.
  \label{cor:insert-factor}
\end{corollary}

For example, the Edelman--Greene correspondence gives a weight-reversing bijection
\[ 
	\RF^\ell(153264) \rightarrow \left( \raisebox{\cellsize}{$\tableau{4 \\ 3 & 5 \\ 2 & 3}$} \times \SSYT_\ell(2,2,1) \right) 
	\bigcup \left( \raisebox{\cellsize}{$\tableau{4 \\ 3 \\ 2 & 3 & 5}$} \times \SSYT_\ell(3,1,1) \right).
 \]
In particular, by the symmetry of Schur functions, we have the following expansion from \cite{EG.1987}.

\begin{corollary}
  The Stanley symmetric function for $w$ may be expressed as
  \begin{equation}
    F_{w}(x) = \sum_{T \in \mathrm{Yam}(w^{-1})} s_{\operatorname{sh}(T)}(x),
  \end{equation}
  where $\mathrm{Yam}(w^{-1})$ is the set of insertion tableaux with $\mathrm{row}(P)$ a reduced word for $w^{-1}$.
\end{corollary}

For example, we have
\begin{equation}
\label{equation.F expansion}
	F_{143625}(x) = s_{(2,2,1)}(x) + s_{(3,1,1)}(x). 
\end{equation}

%%%%%%%%%%%%%%%%%%%%%%%%%%%%%%%%%%%%%%%%%%%%%%%%%%%%%%%%%%%%%%%%
%
\subsection{Crystal operators on reduced factorizations}
\label{sec:RF-crystal}

Following~\cite{MS16}, we are going to define an $A_{\ell-1}$-crystal structure on $\RF^\ell(w)$.
Let $r=r^\ell r^{\ell-1} \cdots r^1 \in \RF^\ell(w)$, where $r^i$ is the $i$th block from the right.
The Kashiwara raising and lowering operators $e_i$ and $f_i$ only act on the blocks $r^{i+1} r^i$. The action is
defined by first bracketing certain letters and then moving an unbracketed letter from one factor to the other.
Let us begin by describing the bracketing procedure. Start with the largest letter $b$ in $r^i$
and pair it with the smallest $a>b$ in $r^{i+1}$. If no such $a$ exists in $r^{i+1}$, then
$b$ is unpaired.  The pairing proceeds in decreasing order on elements of $r^i$, and with each iteration 
previously paired letters of $r^{i+1}$ are ignored. Define
$$
R_i(r^\ell \cdots r^1)=
\{
b\in r^i \mid
b  \text{ is unpaired in the $r^{i+1}r^i$-pairing}
\}
$$
and
$$
L_i(r^\ell \cdots r^1)=
\{
b\in r^{i+1} \mid
b  \text{ is unpaired in the $r^{i+1}r^i$-pairing}
\}\;.
$$

Then $f_i(r^\ell \cdots r^1)$ is defined by replacing the blocks $r^{i+1} r^i$ by $\widetilde r^{i+1} \widetilde r^i$ such that
\newcommand{\cont}{cont}
$$
\widetilde r^i=r^i\backslash\{b\}\quad\text{and} 
\quad \widetilde r^{i+1}=r^{i+1}\cup\{b-t\}
$$
for $b=\min(R_i(r^\ell \cdots r^1))$ and
$t=\min\{j\geqslant 0\mid b-j-1\not\in r^i\}$.
If $R_i(r^\ell \cdots r^1)=\emptyset$, then $f_i(r^\ell \cdots r^1)= 0$.

Similarly, $e_i(r^\ell \cdots r^1)$ is defined by replacing the factors $r^{i+1} r^i$ by $\widetilde r^{i+1} \widetilde r^i$ such that
$$
\widetilde r^i=r^i\cup\{a+s\}\quad\text{and}\quad
r^{i+1}=r^{i+1}\backslash\{a\}
$$ 
for $a=\max(L_i(r^\ell \cdots r^1))$ and
$s=\min\{j\geqslant 0\mid a+j+1\not\in r^{i+1}\}$.
If $L_i(r^\ell \cdots r^1)=\emptyset$, then $e_i(r^\ell \cdots r^1)= 0$.

\begin{example}
Let $(2)(13)(23) \in \RF^3(w)$ for $w= s_2 s_1 s_3 s_2s_3 \in S_4$.
To apply $f_1$ we need to first bracket the letters in $r^1 = 23$ with those in
$r^2 = 13$. The letter 3 in $r^1$ is unbracketed since there is no bigger letter in
$r^2$, but the letter 2 in $r^1$ is bracketed with 3 in $r^2$. Hence 
$b = \min(R_1(r^3 r^2 r^1))=3$ and $t=\min\{j\geqslant 0\mid b-j-1\not\in r^1\}=1$.
Therefore, $f_1((2)(13)(23)) = (2)(123)(2)$.
Similarly, $e_1((2)(13)(23)) = (2)(3)(123)$.
\end{example}

\begin{remark}
In~\cite{MS16}, the Stanley symmetric function $F_w$ is defined using decreasing factorizations of reduced words
of $w$. Here we use increasing factorizations of $w^{-1}$. To relate the two, one needs to revert the reduced
factorizations. The crystal structures are related by interchanging $f_i$ (resp. $e_i$) with $e_{\ell-i}$ (resp. $f_{\ell-i}$).
\end{remark}

\begin{theorem} \cite[Theorem 3.5]{MS16}
\label{theorem.crystal}
The above defined operators $f_i$ and $e_i$ for $1\leqslant i<\ell$ and the weight function $\wt$ define a
$A_{\ell-1}$-crystal structure on $\RF^\ell(w)$.
\end{theorem}

\begin{corollary} \cite{MS16}
  The Stanley symmetric function for $w$ may be expressed as
  \begin{equation}
    F_{w}(x) = \sum_{\substack{r \in \RF^\ell(w^{-1})\\ e_i r = 0 \quad \forall 1\leqslant i < \ell}} s_{\wt(r)}(x).
  \end{equation}
\end{corollary}

For example, the highest weight reduced factorizations for $153264=143625^{-1}$ with $\ell=4$ are
$()(4)(35)(23)$ and $()(4)(3)(235)$ of weights $(2,2,1)$ and $(3,1,1)$, respectively, 
confirming~\eqref{equation.F expansion}.

It turns out that this crystal structure on reduced factorizations relates to the crystal structure on semi-standard 
Young tableaux via the Edelman--Greene correspondence. 

\begin{theorem} \cite[Theorem 4.11]{MS16}
\label{theorem.crystal RF}
  Given $r\in \RF^\ell(w)$, let $P(r)$ denote its Edelman--Greene insertion tableau and $Q(r)$ its Edelman--Greene 
  semi-standard recording tableau, where letters in block $i$ of $r$ are recorded by the letter $i$. Then, if $e_i(r) \neq 0$, 
  we have $P(e_i(r)) = P(r)$ and $Q(e_i(r)) = f_{\ell-i}(Q(r))$. 
\end{theorem}

%%%%%%%%%%%%%%%%%%%%%%%%%%%%%%%%%%%%%%%%%%%%%%%%%%%%%%%%%%%%%%%%
%
\section{Demazure crystal structure for Schubert polynomials}
%
%%%%%%%%%%%%%%%%%%%%%%%%%%%%%%%%%%%%%%%%%%%%%%%%%%%%%%%%%%%%%%%%
\label{sec:RCF}

We review the combinatorial expression of Billey, Jockusch and Stanley~\cite{BJS93} for Schubert polynomials 
in terms of compatible sequences in Section~\ref{sec:schubert} and show that it can be reformulated
in terms of reduced factorizations with a cutoff condition. 
In Section~\ref{sec:EG-weak} we discuss the weak analog of the Edelman--Greene insertion presented in \cite{Ass-R}.
It turns out that the cut-off condition precisely amounts to a Demazure crystal structure as shown in 
Section~\ref{sec:RF-demazure}.

%%%%%%%%%%%%%%%%%%%%%%%%%%%%%%%%%%%%%%%%%%%%%%%%%%%%%%%%%%%%%%%%
%
\subsection{Combinatorics of Schubert polynomials}
\label{sec:schubert}

Schubert polynomials are generalizations of Schur polynomials which represent cohomology classes of Schubert cycles 
in flag varieties. They were first introduced by Bernstein et al.~\cite{BGG.1973} and extensively studied by
Lascoux and Sch\"utzenberger~\cite{LS.1982}.

\begin{definition}[\cite{LS.1982}]
  Given a permutation $w$, the \defn{Schubert polynomial} for $w$ is given by
  \begin{equation}
    \mathfrak{S}_w(x) = \partial_{w^{-1} w_0}(x_1^{n-1} x_2^{n-2} \cdots x_{n-1}),
  \end{equation}
  where $w_0=n n-1 \ldots 21$ is the longest permutation of length $\binom{n}{2}$.
\end{definition}

The first proven combinatorial formula for Schubert polynomials, due to Billey, Jockusch and Stanley~\cite{BJS93}, is 
in terms of compatible sequences for reduced expressions.

\begin{definition}[\cite{BJS93}]
  For $\rho = \rho_1 \ldots \rho_k$ a reduced word, a sequence $\alpha=\alpha_1 \ldots \alpha_k$ of positive integers 
  is \defn{$\rho$-compatible} if $\alpha$ is weakly decreasing, $\alpha_j \leqslant \rho_j$, and $\alpha_j > \alpha_{j+1}$ 
  whenever $\rho_j > \rho_{j+1}$.
  \label{def:compatible}
\end{definition}

For example, seven of the reduced words for $153264$ have compatible sequences as shown in Figure~\ref{fig:compatible}.

\begin{figure}[ht]
  \begin{displaymath}
    \begin{array}{ccccccc}
      45323 & 45232 & 43523 & 43253 & 42352 & 43235 & 42325 \\\hline
      44322 & 44221 & 43322 & 43221 & 42221 & 43222 & 42211 \\
      44321 & 43221 & 43321 &       & 32221 & 43221 & 32211 \\
      44311 & 33221 & 43311 &       &       & 43211 &       \\
      44211 &       & 43211 &       &       & 43111 &       \\
      43211 &       & 42211 &       &       & 42111 &       \\
      33211 &       & 32211 &       &       & 32111 &
    \end{array}
  \end{displaymath}
   \caption{\label{fig:compatible}The compatible sequences for the reduced words for $153264$.}
\end{figure}

\begin{theorem}[\cite{BJS93}]
  The Schubert polynomial $\schubert_w(x)$ indexed by a permutation $w$ is given by
  \begin{equation}
    \schubert_{w}(x) = \sum_{\rho\in R(w^{-1})} \sum_{\alpha\in\mathrm{RC}(\rho)} x^{\alpha} ,
    \label{e:schubert}
  \end{equation}
  where $x^a$ is the monomial $x_1^{a_1} \cdots x_n^{a_n}$.
  \label{thm:schubert-RC}  
\end{theorem}

We may encode compatible sequences for the reduced words as increasing factorizations with an additional cutoff condition.

\begin{definition}
  Given a reduced word $\rho$, an \defn{increasing factorization with cutoff} is an increasing factorization such that in 
  addition the first entry in block $i$ from the right is at least $i$. 
  
  Given a permutation $w$, a \defn{reduced factorization with cutoff} for $w$ is an increasing factorization with cutoff 
  of a reduced word for $w$.
  \label{def:factor-cutoff}
\end{definition}

The set of reduced factorizations with cutoff is denoted by $\RFC(w)$. For example, the reduced factorizations with 
cutoff for $153264$ are shown in Figure~\ref{fig:factor-cutoff}.

\begin{figure}[ht]
  \begin{displaymath}
    \begin{array}{ccccccc}
      (45)(3)(23)() & (45)()(23)(2) & (4)(35)(23)() & (4)(3)(25)(3) & (4)()(235)(2) & (4)(3)(235)() & (4)()(23)(25) \\
      (45)(3)(2)(3) & (4)(5)(23)(2) & (4)(35)(2)(3) &               & ()(4)(235)(2) & (4)(3)(23)(5) & ()(4)(23)(25) \\
      (45)(3)()(23) & ()(45)(23)(2) & (4)(35)()(23) &               &               & (4)(3)(2)(35) &               \\
      (45)()(3)(23) &               & (4)(3)(5)(23) &               &               & (4)(3)()(235) &               \\
      (4)(5)(3)(23) &               & (4)()(35)(23) &               &               & (4)()(3)(235) &               \\
      ()(45)(3)(23) &               & ()(4)(35)(23) &               &               & ()(4)(3)(235) &                
    \end{array}
  \end{displaymath}
   \caption{\label{fig:factor-cutoff}The reduced factorizations with cutoff for $153264$.}
\end{figure}

The weight function on reduced factorizations provides a simple bijection between compatible sequences and 
increasing factorizations with cutoff for a reduced word. For example, compare Figure~\ref{fig:factor-cutoff} 
with Figure~\ref{fig:compatible}.

\begin{proposition}
  The Schubert polynomial $\schubert_w(x)$ is given by
  \begin{equation}
    \schubert_{w}(x) = \sum_{r \in \RFC(w^{-1})} x^{\wt(r)}.
    \label{e:schubert fac}
  \end{equation}
  \label{prop:schubert-RF}    
\end{proposition}

\begin{proof}
  To prove that~\eqref{e:schubert fac} is equivalent to~\eqref{e:schubert}, we show that there is a bijection
  $\bigcup_{\rho \in R(w^{-1})} \mathrm{RC}(\rho) \to \RFC(w^{-1})$.
  Given a compatible sequence $\alpha$ for a reduced word $\rho$, the letter $\rho_i$ belongs to the $a$-th factor from
  the right if $\alpha_i=a$. Due to the condition that $\alpha_j> \alpha_{j+1}$ whenever $\rho_j>\rho_{j+1}$,
  the letters within each factor are weakly increasing. Since the word $\rho$ is reduced, the letters within each factor
  must actually be increasing. Furthermore, since $\alpha_j\leqslant \rho_j$, all letters in the $a$-th factor must be of
  value at least $a$. Conversely, given a reduced factorization with cutoff one can immediately construct the compatible 
  sequence $\alpha$ by setting $\alpha_j=a$ if $\rho_j$ is in factor $a$.
\end{proof}

Reduced factorizations have the advantage of tracking the reduced word along with the weight, making this a more 
natural indexing set for the crystal structure discussed in the next section.

%%%%%%%%%%%%%%%%%%%%%%%%%%%%%%%%%%%%%%%%%%%%%%%%%%%%%%%%%%%%%%%%
%
\subsection{Weak Edelman--Greene correspondence}
\label{sec:EG-weak}

We recall a generalization of the Edelman--Greene correspondence~\cite{Ass-R} that gives the Demazure expansion 
of a Schubert polynomial, parallel to the Schur expansion of a Stanley symmetric function.

Following \cite{Ass-R}, for $P$ a semi-standard Young tableau with strictly increasing rows, define the 
\defn{lift of $P$}, denoted by $\mathrm{lift}(P)$, to be the tableau of key shape obtained by raising each entry in the 
first column of $P$ until it equals its row index, and, once columns $1$ through $c-1$ have been lifted, raising entries in 
column $c$ from top to bottom, maintaining their relative order, placing each entry in the highest available row such that 
there is an entry in column $c-1$ that is strictly smaller.

\begin{definition}[\cite{Ass-R}]
  For $\rho$ a reduced expression, define the \defn{weak insertion tableau} $\widehat{P}(\rho)$ by 
  $\widehat{P}(\rho) = \mathrm{lift}(P(\rho))$, where $P(\rho)$ is the insertion tableau under the Edelman--Greene
  insertion. In addition,  define the \defn{weak recording tableau} $\widehat{Q}(\rho)$ to be the unique standard key 
  tableau of the same key shape as $\widehat{P}(\rho)$ such that $\phi(\widehat{Q}(\rho)) = Q(\rho)$, where $Q(\rho)$ 
  is the Edelman--Greene recording tableau and $\phi$ is the column sorting map.
  \label{def:weak-EG}
\end{definition}

For example, Figure~\ref{fig:EG-weak} constructs the weak insertion tableau (top) and weak recording tableau 
(bottom) for the reduced expression $45232$. Compare this with Figure~\ref{fig:EG}.

\begin{figure}[ht]
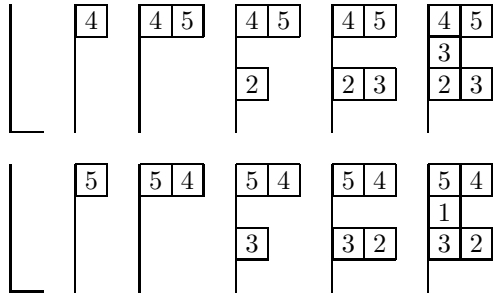

  \begin{displaymath}
    \begin{array}{c@{\hskip\cellsize}c@{\hskip\cellsize}c@{\hskip\cellsize}c@{\hskip\cellsize}c@{\hskip\cellsize}c}
      \vline\tableau{\\ \\ \\ \\\hline} &
      \vline\tableau{4 \\ \\ \\ } &
      \vline\tableau{4 & 5\\ \\ \\ } &
      \vline\tableau{4 & 5\\ \\ 2 \\ } &
      \vline\tableau{4 & 5\\ \\ 2 & 3 \\ } &
      \vline\tableau{4 & 5\\ 3 \\ 2 & 3 \\ } \\ \\
      \vline\tableau{\\ \\ \\ \\\hline} &
      \vline\tableau{5 \\ \\ \\ } &
      \vline\tableau{5 & 4\\ \\ \\ } &
      \vline\tableau{5 & 4\\ \\ 3 \\ } &
      \vline\tableau{5 & 4\\ \\ 3 & 2 \\ } &
      \vline\tableau{5 & 4\\ 1 \\ 3 & 2 \\ } 
    \end{array}
  \end{displaymath}
  \caption{\label{fig:EG-weak}The weak insertion and recording tableaux for the reduced expression $45232$.}
\end{figure}

For $P$ a key tableau, define the \defn{drop of $P$}, denoted by $\mathrm{drop}(P)$, to be the Young tableau 
defined by letting the entries of $P$ fall in their columns while maintaining their relative order. It is clear that 
$\mathrm{drop}(\mathrm{lift}(P))=P$ for any $P$ of partition shape.

\begin{theorem}[\cite{Ass-R}]
  The weak Edelman--Greene correspondence $\rho \mapsto (\widehat{P}(\rho),\widehat{Q}(\rho))$ is a bijection between 
  reduced expressions and all pairs of tableaux $(P,Q)$ such that $P$ and $Q$ have the same key shape, $P$ 
  has increasing rows and columns with $\mathrm{row}(P)$ a reduced word and $\mathrm{lift}(\mathrm{drop}(P))=P$, 
  and $Q$ is a standard key tableau.
  \label{thm:weak-EG}
\end{theorem}

Analogous to the Edelman--Greene correspondence, this extends to a bijection between reduced factorizations 
with cutoff and all pairs of tableaux $(P,Q)$ such that $P$ and $Q$ have the same key shape, $P$ is increasing with 
$\mathrm{row}(P)$ a reduced word and $\mathrm{lift}(\mathrm{drop}(P))=P$, and $Q$ is a semi-standard key tableau. 
For example, the recording tableau for the reduced factorization $(4)(5)(23)(2)$ is constructed in 
Figure~\ref{fig:EG-weak-factor}.

\begin{figure}[ht]
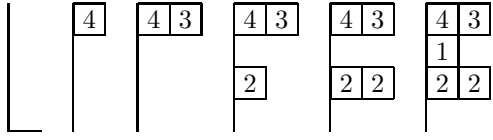

  \begin{displaymath}
    \begin{array}{c@{\hskip\cellsize}c@{\hskip\cellsize}c@{\hskip\cellsize}c@{\hskip\cellsize}c@{\hskip\cellsize}c}
      \vline\tableau{\\ \\ \\ \\\hline} &
      \vline\tableau{4 \\ \\ \\ } &
      \vline\tableau{4 & 3\\ \\ \\ } &
      \vline\tableau{4 & 3\\ \\ 2 \\ } &
      \vline\tableau{4 & 3\\ \\ 2 & 2 \\ } &
      \vline\tableau{4 & 3\\ 1 \\ 2 & 2 \\ }
    \end{array}
  \end{displaymath}
  \caption{\label{fig:EG-weak-factor}The weak recording tableau for the reduced factorization $(4)(5)(23)(2)$.}
\end{figure}

\begin{corollary}
  The correspondence $r \mapsto (\widehat{P}(r),\widehat{Q}(r))$ is a weight-preserving bijection between 
  reduced factorizations and all pairs of tableaux $(P,Q)$ such that $P$ and $Q$ have the same key shape, $P$ 
  is increasing with $\mathrm{row}(P)$ a reduced word and $\mathrm{lift}(\mathrm{drop}(P))=P$, and $Q$ is a 
  semi-standard key tableau. 
  \label{cor:weak-factor}
\end{corollary}

\begin{proof}
  Theorem~\ref{thm:weak-EG} is proved in~\cite[Theorem~5.16]{Ass-R} using the standard key tableau. To get the 
  semi-standard case, we appeal to~\cite[Proposition~2.6]{Ass-H} where it is shown that the fundamental slide polynomial, 
  defined in~\cite{AS17}, associated to a standard key tableau is the sum of monomials associated to the semi-standard 
  key tableaux that standardize to it. As shown in~\cite[Theorem~2.4]{Ass-R}, the fundamental slide polynomial associated 
  to a reduced expression is the sum of monomials associated to the corresponding compatible sequences. The result 
  follows from the bijection between compatible sequences and increasing factorizations. 
\end{proof}

For example, the weak Edelman--Greene correspondence gives a weight-preserving bijection
\[ 
	\RFC(153264) \rightarrow \left( \raisebox{\cellsize}{$\vline\tableau{4 & 5\\ 3 \\ 2 & 3 \\ & }$} \times \SSKT(0,2,1,2) \right)
	 \bigcup \left( \raisebox{\cellsize}{$\vline\tableau{4 \\ 3 \\ 2 & 3 & 5 \\ & }$} \times \SSKT(0,3,1,1) \right). 
\]
In particular, we have the following expansion from \cite{Ass-R}.

\begin{corollary}[\cite{Ass-R}]
  The Schubert polynomial for $w$ may be expressed as
  \begin{equation}
    \schubert_{w}(x) = \sum_{T \in \mathrm{Yam}(w^{-1})} \key_{\wt(T)}(x),
  \end{equation}
  where $\mathrm{Yam}(w^{-1})$ is the set of increasing tableaux of key shape with $\mathrm{row}(P)$ a reduced 
  word for $w^{-1}$ and $\mathrm{lift}(\mathrm{drop}(P))=P$.
\end{corollary}

For example, we have
\[ 
	\schubert_{143625}(x) = \key_{(0,2,1,2)}(x) + \key_{(0,3,1,1)}(x).
\]

%%%%%%%%%%%%%%%%%%%%%%%%%%%%%%%%%%%%%%%%%%%%%%%%%%%%%%%%%%%%%%%%
%
\subsection{Demazure crystal operators on reduced factorizations with cutoff}
\label{sec:RF-demazure}

Since $\RFC(w) \subseteq \RF^n(w)$ for $w\in S_n$, we can restrict the crystal operators $f_i$ and $e_i$ 
on reduced factorizations to $\RFC(w)$ by defining $f_i(r)$ as in Section~\ref{sec:RF-crystal} if $f_i(r) \in \RFC(w)$ 
and $f_i(r)=0$ otherwise and similarly for $e_i$. An example is given in Figure~\ref{fig:RC-143625}. 

We will show in this section that this amounts to a union of Demazure crystal 
structures. We begin with an analog of Theorem~\ref{theorem.crystal RF}.

\begin{theorem} 
\label{theorem.weak EG}
  Given $r\in \RFC(w)$ for $w\in S_n$, denote by $\widehat{P}(r)$ the weak Edelman--Greene insertion tableau and 
  by $\widehat{Q}(r)$ the weak Edelman--Greene recording tableau, where letters in block $i$ of $r$ are recorded by 
  the letter $i$. Then, if $e_i(r) \neq 0$, we have $\widehat{P}(e_i(r)) = \widehat{P}(r)$ and 
  $\widehat{Q}(e_i(r)) = e_i(\widehat{Q}(r))$ for $1\leqslant i<n$.
\end{theorem}

\begin{proof}
By Theorem~\ref{theorem.crystal RF} we have $P(e_i(r))=P(r)$ and 
$Q(e_i(r)) = f_{n-i}(Q(r))$, where $P$ and $Q$ are the Edelman--Greene insertion and recording tableaux, respectively.
By Definition~\ref{def:weak-EG}, we have $\widehat{P}(r)=\mathrm{lift}(P(r))$, which proves $\widehat{P}(e_i(r)) 
= \widehat{P}(r)$. Again by Definition~\ref{def:weak-EG}, we have $\phi(\widehat{Q}(r)) = Q(r)$. By 
Lemma~\ref{lem:commute}, we have $\phi(e_i \widehat{Q}(r)) = f_{n-i} \phi(\widehat{Q}(r))
=f_{n-i} Q(r)$, proving that $\widehat{Q}(e_i(r)) = e_i(\widehat{Q}(r))$.
\end{proof}

\begin{figure}[t]
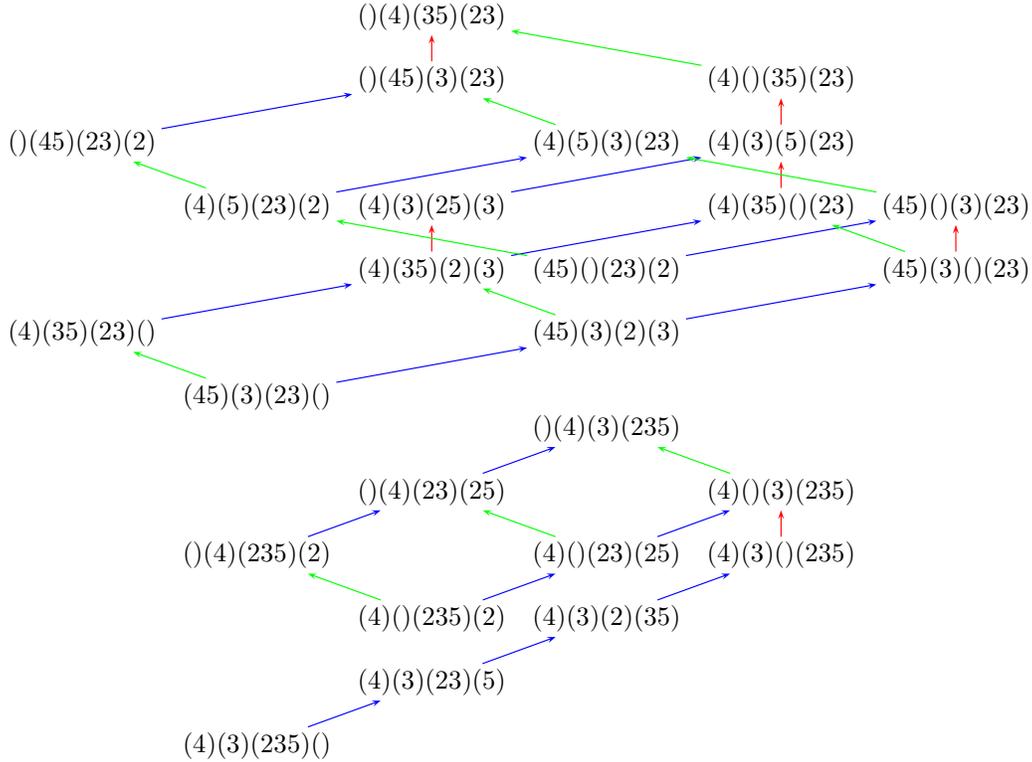

  \begin{displaymath} \begin{array}{cccccc}
      & & \rnode{b1}{()(4)(35)(23)} & & \\[\cellsize]
      & & \rnode{b2}{()(45)(3)(23)} & & \rnode{c2}{(4)()(35)(23)} & \\[\cellsize]
      \rnode{a3}{()(45)(23)(2)} & & & \rnode{c3}{(4)(5)(3)(23)} & \rnode{C3}{(4)(3)(5)(23)} & \\[\cellsize]
      & \rnode{b4}{(4)(5)(23)(2)} & \rnode{B4}{(4)(3)(25)(3)} & & \rnode{c4}{(4)(35)()(23)} & \rnode{d4}{(45)()(3)(23)} \\[\cellsize]
      & & \rnode{b5}{(4)(35)(2)(3)} & \rnode{c5}{(45)()(23)(2)} & & \rnode{d5}{(45)(3)()(23)} \\[\cellsize]
      \rnode{a6}{(4)(35)(23)()} & & & \rnode{c6}{(45)(3)(2)(3)} & & \\[\cellsize]
      & \rnode{b7}{(45)(3)(23)()} & & & & \\
      & & & \rnode{xc1}{()(4)(3)(235)} & & \\[\cellsize]
      & & \rnode{xb2}{()(4)(23)(25)} & & \rnode{xd2}{(4)()(3)(235)} & \\[\cellsize]
      & \rnode{xa3}{()(4)(235)(2)} & & \rnode{xc3}{(4)()(23)(25)} & \rnode{xd3}{(4)(3)()(235)} & \\[\cellsize]
      & & \rnode{xb4}{(4)()(235)(2)} & \rnode{xc4}{(4)(3)(2)(35)} & & \\[\cellsize]
      & & \rnode{xb5}{(4)(3)(23)(5)} & & & \\[\cellsize]
      & \rnode{xa6}{(4)(3)(235)()} & & & &
    \end{array}
    \psset{nodesep=2pt,linewidth=.1ex}
    \ncline[linecolor=red]{<-}  {b1}{b2} %\ncput*{2}
    \ncline[linecolor=green]{<-}{b1}{c2} %\ncput*{3}
    \ncline[linecolor=blue]{<-} {b2}{a3} %\ncput*{1}
    \ncline[linecolor=red]{<-}  {c2}{C3} %\ncput*{2}
    \ncline[linecolor=green]{<-}{b2}{c3} %\ncput*{3}
    \ncline[linecolor=blue]{<-} {c3}{b4} %\ncput*{1}
    \ncline[linecolor=blue]{<-} {C3}{B4} %\ncput*{1}
    \ncline[linecolor=red]{<-}  {C3}{c4} %\ncput*{2}
    \ncline[linecolor=green]{<-}{a3}{b4} %\ncput*{3}
    \ncline[linecolor=green]{<-}{c3}{d4} %\ncput*{3}
    \ncline[linecolor=blue]{<-} {c4}{b5} %\ncput*{1}
    \ncline[linecolor=blue]{<-} {d4}{c5} %\ncput*{1}
    \ncline[linecolor=red]{<-}  {B4}{b5} %\ncput*{2}
    \ncline[linecolor=red]{<-}  {d4}{d5} %\ncput*{2}
    \ncline[linecolor=green]{<-}{b4}{c5} %\ncput*{3}
    \ncline[linecolor=green]{<-}{c4}{d5} %\ncput*{3}
    \ncline[linecolor=blue]{<-} {b5}{a6} %\ncput*{1}
    \ncline[linecolor=blue]{<-} {d5}{c6} %\ncput*{1}
    \ncline[linecolor=green]{<-}{b5}{c6} %\ncput*{3}
    \ncline[linecolor=blue]{<-} {c6}{b7} %\ncput*{1}
    \ncline[linecolor=green]{<-}{a6}{b7} %\ncput*{3}
    \ncline[linecolor=blue]{<-} {xc1}{xb2} %\ncput*{1}
    \ncline[linecolor=green]{<-}{xc1}{xd2} %\ncput*{3}
    \ncline[linecolor=blue]{<-} {xb2}{xa3} %\ncput*{1}
    \ncline[linecolor=green]{<-}{xb2}{xc3} %\ncput*{3}
    \ncline[linecolor=blue]{<-} {xd2}{xc3} %\ncput*{1}
    \ncline[linecolor=red]{<-}  {xd2}{xd3} %\ncput*{2}
    \ncline[linecolor=green]{<-}{xa3}{xb4} %\ncput*{3}
    \ncline[linecolor=blue]{<-} {xc3}{xb4} %\ncput*{1}
    \ncline[linecolor=blue]{<-} {xd3}{xc4} %\ncput*{1}
    \ncline[linecolor=blue]{<-} {xc4}{xb5} %\ncput*{1}
    \ncline[linecolor=blue]{<-} {xb5}{xa6} %\ncput*{1}    
  \end{displaymath}
  \caption{\label{fig:RC-143625}The Demazure crystal structure on $\RFC(153264)$, with edges $e_1 
  \color{blue}\nearrow$, $e_2 \color{red}\uparrow$, $e_3 \color{green}\nwarrow$.}
\end{figure}

By Proposition~\ref{prop:schubert-RF}, combinatorial objects underlying the Schubert polynomials
$\schubert_{w^{-1}}(x)$ are the reduced factorizations with cutoff $\RFC(w)$. On the other hand,
$\RF^n(w)$ are combinatorial objects underlying the Stanley symmetric polynomials $F_{w^{-1}}(x)$
by Definition~\ref{def:stanley}. By Theorem~\ref{theorem.crystal}, there is a crystal structure
on $\RF^n(w)$. Now we show that $\RFC(w)$ admits a Demazure crystal structure.

\begin{theorem}
\label{theorem.main}
	The operators $f_i$ and $e_i$ for $1\leqslant i<n$ define a Demazure crystal structure on
	$\RFC(w)$. More precisely,
	\[
		\RFC(w) \cong \bigcup_{\substack{r\in \RFC(w)\\ e_i r = 0 \quad \forall 1\leqslant i<n}}
		B_{w(r)}(\wt(r)),
	\]
	where $w(r)$ is the shortest permutation that sorts $\operatorname{sh}(\widehat{P}(r))$.
\end{theorem}

\begin{proof}
By Theorem~\ref{theorem.weak EG}, the crystal operators on reduced factorizations under weak Edelman--Greene
insertion intertwine with the crystal operators on key tableaux. On the other hand, by Theorem~\ref{theorem.key demazure}
the crystal operators on key tableaux form a Demazure crystal.
\end{proof}

For example, the highest weight elements in $\RFC(153264)$ are $()(4)(35)(23)$ and $()(4)(3)(235)$
(see Figure~\ref{fig:RC-143625}), so that as Demazure crystals
\[
	\RFC(153264) \cong B_{s_1s_3s_2s_3}(2,2,1) \cup B_{s_1s_2s_3}(3,1,1).
\]

\begin{corollary}
  The Schubert polynomial for $w\in S_n$ may be expressed as
  \begin{equation}
    \schubert_{w}(x) = \sum_{\substack{r \in \RFC(w^{-1})\\ e_i r = 0 \quad \forall 1\leqslant i < n}} 
    \key_{\operatorname{sh}(\widehat{P}(r))}(x).
  \end{equation}
\end{corollary}

%%%%%%%%%%%%%%%%%%%%%%%%%%%%%%%%%%%%%%%%%%%%%%%%%%%%%%%%%%%%
%
%  Bibliography
%
%%%%%%%%%%%%%%%%%%%%%%%%%%%%%%%%%%%%%%%%%%%%%%%%%%%%%%%%%%%%

\bibliographystyle{amsalpha} 
\bibliography{demazure}

\end{document}